\documentclass[a4paper]{article}
\usepackage{amsmath,amssymb,amsthm,amsfonts}
\usepackage{enumerate,color, graphicx}
\usepackage{supertabular}
\usepackage{booktabs}
\usepackage{url}
\usepackage{xcolor}
\usepackage{colortbl}
\usepackage{subfig}
\newcommand{\R}{\ensuremath{\mathbb{R}}}

\newcommand{\N}{\ensuremath{\mathbb{N}}}

\definecolor{dgreen}{rgb}{0.,0.6,0.}

\numberwithin{equation}{section}

\newtheorem{thm}{Theorem}[section]
\newtheorem{prop}[thm]{Proposition}
\newtheorem{cor}[thm]{Corollary}
\newtheorem{lem}[thm]{Lemma}

\newtheorem{rem}[thm]{Remark}

\title{Courant-sharp Robin eigenvalues for the square \\
-- the case  with small Robin parameter--}

\author{{K. Gittins
\footnote{Universit\'e de Neuch\^atel, Institut de Math\'ematiques, Rue Emile-Argand 11, CH-2000 Neuch\^atel.
Email: \texttt{katie.gittins@unine.ch}}\,
and B. Helffer
\footnote{Laboratoire de Math\'ematiques Jean Leray, Universit\'e de Nantes, 2 rue de la Houssini\`ere, 44 322 Nantes CEDEX 3 - FRANCE.
Email: \texttt{Bernard.Helffer@univ-nantes.fr}}}
}

\date{\today}
\begin{document}

\maketitle
\begin{abstract}	
This article is the continuation of our first work on the determination of the cases where there is
equality in Courant's Nodal Domain theorem in the case of a Robin boundary condition (with Robin parameter $h$).
For the square, our first paper focused on the case where $h$ is large and extended results that were obtained by Pleijel, B\'erard-Helffer, for the problem with a Dirichlet boundary condition.
There, we also obtained some general results about the behaviour of the nodal structure (for planar domains) under a small deformation of $h$, where $h$ is positive and not close to $0$.
In this second paper, we extend results that were obtained by Helffer--Persson-Sundqvist for the Neumann problem to the case where $ h>0$ is small.\\

\end{abstract}

\paragraph{MSC classification (2010):}    35P99, 58J50, 58J37.


\newpage
\tableofcontents
\newpage
\section{Introduction}

Consider a bounded, connected, open set $\Omega \subset \R^m$, $m \geq 2$, with Lipschitz boundary.
Let $h \in \R$, $h \geq 0$.
We consider the Robin eigenvalues of the Laplacian on $\Omega$ with parameter $h$.
That is the values $\lambda_{k,h}(\Omega) \in \R$, $k \in \N$, such that there exists a function $u_k \in H^1(\Omega)$ that satisfies
\begin{align*}
-\Delta u_k(x) &= \lambda_{k,h}(\Omega)u_k(x)\,, \quad x \in \Omega\,, \notag \\
\frac{\partial}{\partial \nu} u_k(x) &+ h\, u_k(x) = 0\,,  \quad x \in \partial\Omega\,, 
\end{align*}
where $\nu$ is the outward-pointing unit normal to $\partial \Omega$.

We recall that the corresponding spectrum is monotonically increasing with respect to $h$ for $h\in [0,+\infty)$,
by the minimax principle. In particular, the Robin eigenvalues with $h=0$ correspond to
the Neumann eigenvalues, $\lambda_k^N(\Omega)$, while those with $h=+\infty$ correspond to the Dirichlet eigenvalues
$\lambda_k^D(\Omega)$.

We consider the Courant-sharp Robin eigenvalues of $\Omega$.
That is, those Robin eigenvalues $\lambda_{k,h}(\Omega)$ that have a corresponding eigenfunction which has exactly $k$
nodal domains, and hence achieves equality in Courant's Nodal Domain theorem. As for the Dirichlet and Neumann eigenvalues,
$\lambda_{1,h}(\Omega)$ and $ \lambda_{2,h}(\Omega)$ are Courant-sharp for all $h \geq 0$.

The question that we first considered in \cite{GH} is whether it is possible to follow the Courant-sharp (Neumann)
eigenvalues with $h=0$ to Courant-sharp (Dirichlet) eigenvalues as $h \to + \infty$.
There we analysed the situation where $h$ is large. Our aim in this paper is to analyse the case where $h$ is small.

As in \cite{GH}, we consider the particular example where $\Omega$ is a square in $\R^2$ of side-length $\pi$  which we denote by $S$.
There, we were able to treat the problem asymptotically as $h\rightarrow +\infty$, corresponding to the Dirichlet limit.
Moreover, we showed that for $h$ large enough, the only Courant-sharp Robin eigenvalues are for $k=1,2,4$
(see also \cite{Pl,BH1} where the Dirichlet case was treated):
 \begin{thm}\label{thm:hlarge}
 There exists $h_1>0$ such that for $h\geq h_1$, the Courant-sharp cases for the Robin problem are the same
 as those for $h=+\infty\,$.
 \end{thm}

We also obtained the following $h$-independent result:
\begin{thm}\label{prop:p1} Let $h\geq 0$.
If $\lambda_{k,h}(S)$ is an eigenvalue of $S$ with $k \geq 520$, then it is not Courant-sharp.
\end{thm}

For the square with a Neumann boundary condition, it was shown in \cite{HPS1} that the only Courant-sharp
Neumann eigenvalues are for $k=1,2,4,5,9$. Hence it is natural to ask whether this result also holds for $h\geq 0$ small.
The goal of this paper is to prove the following theorem which was conjectured in \cite{GH}.

\begin{thm}\label{thm:hsmall}
 There exists $h_0>0$ such that for $0<h \leq h_0$, the Courant-sharp cases for the Robin problem are the same, except the fifth one, as those for $h=0$\,.
\end{thm}

In \cite{GH}, we showed that there exists $h_9^* > 0$ such that $\lambda_{9,h}$ is Courant-sharp
for $h \leq h_9^*$ and is not Courant-sharp for $h > h_9^*$.

Here we show that the fifth Robin eigenvalue $\lambda_{5,h}$ is not Courant-sharp for any $h>0$.
It is interesting that there is stability of the Courant-sharp property under a small perturbation of $h$ large,
whereas for $h>0$ small there is no stability for $k=5$.
We remark that we do not give any general information about the Courant-sharp property for the Robin
eigenvalues with $h \in (h_0,h_1)$.

We now outline the strategy of the proof of Theorem~\ref{thm:hsmall}.
By making use of the fact that $h$ is small, we reduce the number of potential Courant-sharp cases that have to be checked from $k \leq 519$, as in \cite{GH}, to $k \leq 208$ as in \cite{HPS1} (see Section~\ref{S:hsmall}).

The strategy of \cite{HPS1} is then to use symmetry properties of the Neumann eigenfunctions and an argument
due to Leydold to further reduce the potential Courant-sharp candidates.
Since the Robin eigenfunctions satisfy analogous symmetry properties (see Section~\ref{s4}), corresponding arguments can be applied to the Robin eigenfunctions when $h$ is sufficiently small (see Section~\ref{s5}).

As in \cite{HPS1}, there are some eigenvalues for which these symmetry arguments do not allow us to conclude whether or not they are Courant-sharp.
In the Robin case, there are also some additional cases that cannot be treated by these symmetry arguments.
In order to deal with the remaining cases, we develop some corresponding results to \cite{GH} for $h$ small.
In particular, under certain hypotheses, we show that for $h>0$ small, under a small perturbation of $h$, the number
of nodal domains cannot increase (see Section~\ref{s6} and Section~\ref{s7}). So if a Neumann eigenvalue is not Courant-sharp, then the corresponding Robin eigenvalue is not Courant-sharp for $h$ sufficiently small.
We apply these results in Section~\ref{s9} to eliminate all but two of the remaining cases.

There are then two outstanding cases for which these arguments do not apply and we have to do a detailed analysis.
One of these cases is $\lambda_{5,h}$ as it is Courant-sharp for $h=0$ and so the fact that the number of nodal domains does not increase for $h$ small is not sufficient to show that this eigenvalue is not Courant-sharp for $h>0$ small.  For the other outstanding case, we do not prove that the number of nodal domains is decreasing but we show that this number does not increase too much under a small perturbation of $h$.
These remaining cases are analysed in Section~\ref{s10} and Section~\ref{s11}.

\textbf{Acknowledgements:}\\
 We are very grateful to Thomas Hoffmann-Ostenhof for useful remarks, to Mikael P. Sundqvist for his communication of figures, and to Alexander Weisse for introducing us to the mathematics software system ``SageMath'' and helping us to produce some graphs of the Robin eigenvalues of the square. KG acknowledges support from the Max Planck Institute for Mathematics, Bonn, from October 2017 to July 2018. BH acknowledges the support of the  Mittag-Leffler Institute, Djursholm, where part of this work has been achieved.

\section{The eigenfunctions of the Robin Laplacian for a square}\label{s4}
\subsection{The $1D$ case in $(-\frac \pi 2,\frac \pi 2)$}
\subsubsection{{The case} $h\geq 0$}
The eigenvalues are determined by
 the unique solution $\alpha_n (h)$ in $[n \pi,(n+1) \pi)$ of
\begin{equation}\label{eq:alphaneven}
\alpha \tan \frac \alpha 2 = h \ell\, ,
\end{equation}
 for $n$ even, and the unique solution $\alpha_n (h)$ in $[n \pi,(n+1) \pi)$ of
\begin{equation}\label{eq:alphanodd}
\frac {\alpha}{h\ell} = -  \tan  \frac \alpha 2\, ,
\end{equation}
 for $n$ odd.\\
The corresponding eigenvalue is $\lambda_{n+1} (h)= \frac{\alpha_n(h)^2}{\pi^2}$ and
a basis of corresponding eigenfunctions is given for $n$ even by
$$
u_{n,h} (x) = \cos \left( \frac{\alpha_n(h) x}{\pi} \right)  \, ,
$$
and for $n$ odd by
$$
u_{n,h}(x) = \sin \left( \frac{\alpha_{n}(h) x}{\pi} \right) \,.
$$
Note that
 $$
 \alpha_n(h) \geq \alpha_n(0)\,.
 $$
 Moreover, when $h=0$, we have
 $$
  \alpha_n (0) = n\pi\,,
 $$
 and recover the standard basis
 $$
 u_{n}^{Ne} (x) = \cos (n x) \mbox{ for } n \mbox{ even},\mbox{ and }   u_{n}^{Ne} (x) = \sin  (n x) \mbox{ for } n \mbox{ odd}.
 $$
 of the Neumann problem.
 We recall that  the eigenfunctions  $u_i$ are alternately symmetric and antisymmetric:
\begin{equation*}
u_i (-x) = (-1)^{i} u_{i}(x)\,.
\end{equation*}

\subsubsection{The case $h<0$}

For $h<0$, $|h|$ small enough and $n \geq1$, there are solutions $\alpha_n$ of \eqref{eq:alphaneven} or \eqref{eq:alphanodd}. But for $h < 0$, the first Robin eigenvalue is negative (see (6) of \cite{FK}), so one should also look for energies
with a purely imaginary $\alpha = i \beta$ and consider
\begin{equation}\label{defbeta}
\beta \tanh \frac  \beta 2  = - h \ell\,.
\end{equation}
The corresponding energy is $-\beta^2$.
This gives one additional solution with $\beta >0$ corresponding to the ground state energy.
This solution is unique and, for $|h|$ small enough, is the only negative eigenvalue.
The ground state can be defined by
\begin{equation*}
u(x) = \, \cosh \left( \frac{\beta x}{\pi} \right)\,,
\end{equation*}
with $\beta$ defined by \eqref{defbeta}.\\

\subsection{$2D$-case}\label{ss4.1}
For the square $S =(-\frac \pi 2,\frac \pi 2)^2$\,,
an orthonormal basis  of eigenfunctions for the Robin problem is given by
\begin{equation*}
u_{i,j} (x,y) =  u_i(x)u_j(y)
\end{equation*}
where, for $n\in \mathbb N$,  $u_n$ is the $(n+1)$-st eigenfunction of the Robin problem in $(-\frac \pi 2,\frac \pi 2)$.
We denote by $\lambda_{i,j}$ the corresponding eigenvalue $ \pi^{-2} (\alpha_i^2 + \alpha_j^2)$.
Very often, when $i\neq j$, we have to analyse the nodal set of the family
\begin{equation}\label{defphitheta}
\Phi_{\theta,i,j} = \cos \theta \, u_i(x) u_j(y) + \sin \theta \, u_i(y) u_j (x)\,.
\end{equation}
When $\lambda_{i,j}$ has exact multiplicity $2$, this family generates all the corresponding eigenspace.

We now observe  that the following lemma holds.
\begin{lem}
\label{lem:thetareduction}
Let $h\geq 0$. The number of nodal domains of $\Phi_{\theta,p,q}$ is the same as the number of nodal domains of  $\Phi_{\frac \pi 2-\theta,p,q}$.
If  $p+q$ is odd, the number of nodal domains of $\Phi_{\theta,p,q}$ is the same as the number of nodal domains of $\Phi_{\pi-\theta, p,q}$.
\end{lem}

\begin{proof}
For the first statement, we observe that
\begin{equation}\label{eq:2.8}
\Phi_{\frac \pi 2-\theta, p,q}(x,y)
 = \Phi_{\theta, p,q}(y,x) =  \Phi_{\theta,q,p}(x,y) \,.
\end{equation}
For the second statement,  we can assume, without loss of generality, that $p$ is even and $q$ is odd. Then
the statement follows directly from the relation (for any $(x,y)\in (-\frac \pi 2,\frac \pi 2)^2$):
\begin{equation}
\label{eq:thetared1}\Phi_{\pi-\theta, p,q}(x,-y)
 = \Phi_{\theta,p,q}(x,y)\,.
\end{equation}
\end{proof}
\begin{rem}\label{remut}
When $p+q$ is odd, this allows us to reduce the analysis to $\theta \in [0,\frac \pi 4]$.
\end{rem}

In what follows, we consider $ h \geq 0$.

\subsection{Particular cases $k=1,2,3,4,9$}\label{ss:2.3}
We recall from \cite{GH} that $\lambda_{1,h}$, $\lambda_{2,h}$ and $\lambda_{4,h}$
are Courant-sharp for any $h\in [0,+\infty]\,$.
We have also proved the following proposition in \cite{GH}.
\begin{prop}\label{prop:lambda9}
There exists $h_9^* >0$ such that $\lambda_{9,h}$ is Courant-sharp for $0\leq h \leq h_9^*$ and not Courant-sharp for $h>  h_9^*$.
\end{prop}

Hence, from this point onwards, we are only interested in the remaining eigenvalues, i.e. in the eigenvalues $\lambda_{n,h}(S)$ with $n > 4$  and $n \neq 9$.
Note that, due to the monotonicity of the Robin eigenvalues with respect to $h$, we have  for $n \geq 4\,$,
\begin{equation*}
\lambda_{n,h}(S) \geq \lambda_{4,h}(S) \geq \lambda_{4,0}(S) = 2\, .
\end{equation*}

\subsection{Symmetry properties}\label{s8}
 We recall the following symmetry properties of the Robin eigenfunctions from \cite[Section 2.3]{GH}.
As was mentioned in \cite{GH}, the use of such symmetries was fruitful in the Neumann case, \cite{HPS1},
by invoking an argument due to Leydold, \cite{Ley}.

In 2D, we consider the possible symmetries of a general eigenfunction associated with the eigenvalues $\lambda_{n,h}$ which reads,
\begin{equation}\label{eq:2.9}
   u(x,y) = \sum_{i,j : \lambda_{n,h}(S) = \pi^{-2}(\alpha_{i}^2 + \alpha_{j}^2)}
    a_{ij}\,  u_{i}(x)u_{j}(y)\,,
 \end{equation}
 where $u_{i}$ is the $(i+1)$-st eigenfunction of the $h$-Robin problem in $(-\frac \pi 2,\frac \pi 2)$.\\
 By  considering the transformation $(x,y) \mapsto (-x, -y)$, we obtain
 \begin{equation}\label{eq:2.9b}
    u(-x, - y) = \sum_{i,j : \lambda_{n,h}(S) = \pi^{-2}(\alpha_{i}^2 + \alpha_{j}^2)}
    a_{ij}\, (-1)^{i+j}  u_{i}(x)u_{j}(y)\,.
 \end{equation}

 \begin{rem}\label{remsym}
 We note that if $(i+j)$ is odd for any pair $(i,j)$ such that $\lambda_{n,h}(S) = \pi^{-2}(\alpha_{i}^2 + \alpha_{j}^2)$, then we get by \eqref{eq:2.9b}, $u(-x,-y) = - u(x,y)$ and as a consequence $u$ has an even number of nodal domains.
 \end{rem}

\section{Former bounds for the number of Courant-sharp Robin eigenvalues of a square}\label{s3}
In this section, we recall the $h$-independent bounds from \cite{GH} and the corresponding Neumann bounds from \cite{HPS1}.

\subsection{Lower bound for the Robin counting function}

Recall that for $\lambda>0$, the Robin counting function for the corresponding
eigenvalues of $\Omega$ is defined as
\begin{equation*}
  N_{\Omega}^{R,h}(\lambda) := \#\{k \in \N : \lambda_{k,h}(\Omega) < \lambda\}.
\end{equation*}
The Neumann counting function $ N_{\Omega}^{Ne}(\lambda)$ corresponds to the case $h=0$.
We recall  that the Robin eigenvalues are monotone with respect to $h \in [0,+\infty)$

When $\Omega=S$, we have
\begin{equation}
\label{eq:WeylN}
 \frac{\pi}{4}\lambda + 2\lfloor\sqrt{\lambda}\rfloor +1 \geq  N_{S}^{Ne} (\lambda)>\frac{\pi}{4}\lambda\,,
\end{equation}
and by comparison with the Dirichlet problem, we also have, for $\lambda \geq 2$,
\begin{equation}
  N_{S}^{R,h}(\lambda)
  > \frac{\pi}{4}\lambda - 2\sqrt{\lambda} + 1\,. \label{eq:R1}
\end{equation}

With $\lambda=\lambda_{n,h}>\lambda_{n-1,h}$ and $\Psi$ an associated
eigenfunction, \eqref{eq:R1} becomes
\begin{equation}
\label{eq:R2}
n > \frac{\pi}{4}\lambda_{n,h} - 2\sqrt{\lambda_{n,h}} + 2\,.
\end{equation}

We now work analogously to the proof of Proposition 2.1 in \cite{HPS1}  (see also Section 3 of \cite{GH}).
Denote by $\Omega^{\text{inn}}$ the union of nodal domains of $\Psi$ whose boundaries do not
touch the boundary of $\Omega$ (except at isolated points), and $\mu^{\text{inn}}(\Psi)$ the number of
nodal domains of $\Psi$ in $\Omega^{\text{inn}}$. Similarly denote by $\Omega^{\text{out}}$ the nodal
domains in $\Omega \setminus \Omega^{\text{inn}}$, and $\mu^{\text{out}}(\Psi)$ the number of
nodal domains of $\Psi$ in $\Omega^{\text{out}}$.
We have that $$\mu^{\text{inn}}(\Psi)= \mu(\Psi)- \mu^{\text{out}}(\Psi)$$ and we require an upper bound
for $\mu^{\text{out}}(\Psi)$.\\
In the case, of the square, we have proven  the following lemma in \cite{GH}.
\begin{lem}\label{lemmuout}
Let $\lambda$ be a Robin eigenvalue of $S$ with $h < +\infty$. If $\psi$ is a Robin eigenfunction associated to $\lambda$, then
\begin{equation*}
  \mu^{\text{out}}( \psi) \leq 4 \sqrt{\lambda}\,.
\end{equation*}
\end{lem}

\subsection{Upper bound for Courant-sharp Robin eigenvalues of a square}
By Lemma~\ref{lemmuout}, we have
\begin{equation}
\label{eq:R3}
\mu^{\text{inn}}(\Psi)\geq \mu(\Psi)-4\sqrt{\lambda_{n,h}}\,.
\end{equation}

Now, $\Omega^{\text{inn}}=\bigcup_i \omega^{\text{inn}}_{i}$ is a finite
union of nodal domains for $\Psi$.  Assuming that $\Omega^{\text{inn}}$ is not empty,   we get, on each $\omega^{\text{inn}}_{i}$, by Faber-Krahn (see~\cite{Pl}), that
\begin{equation}\label{eq:R3a}
\frac{A(\omega^{\text{inn}}_i)}{\pi\mathbf{j}^2}\geq \frac{1}{\lambda_{n,h}}\,,
\end{equation}
where $A(\omega^{\text{inn}}_i)$ denotes the area of
$\omega^{\text{inn}}_i$ and $\mathbf{j}$ denotes
the first positive zero of the Bessel function $J_0$. Adding, and
invoking~\eqref{eq:R3}, we find
\begin{equation*}
\frac{\pi}{\mathbf{j}^2}=\frac{A(S)}{\pi\mathbf{j}^2}>
\frac{A(\Omega^{\text{inn}})}{\pi \mathbf{j}^2}\geq \frac{\mu^{\text{inn}}(\Psi)}{\lambda_{n,h}}\geq \frac{\mu(\Psi)-4\sqrt{\lambda_{n,h}}}{\lambda_{n,h}}\,,
\end{equation*}
from which we extract
\begin{equation}\label{eq:R4}
\frac{\pi}{\mathbf{j}^2} \geq \frac{\mu(\Psi)-4\sqrt{\lambda_{n,h}}}{\lambda_{n,h}}\,.
\end{equation}
Due to \eqref{eq:R3}, this inequality is still true if $\Omega^{\text{inn}}$ is empty.

So, similarly to  \cite{BH1}  
and \cite[Proposition 2.1]{HPS1}, we obtain the following.
\begin{prop}
\label{lem:R1}
Any Courant-sharp Robin eigenvalue  $\lambda_{n,h}$ of $S$ has $n \leq 519$.
\end{prop}
\subsection{Recap of Helffer-Persson-Sundqvist for Neumann}\label{ss:3.4}
 We recall that in the Neumann case, \cite{HPS1}, the proof goes as follows.
Assume that $(\lambda_n,\Psi_n)$ is a Courant-sharp eigenpair. Courant's Nodal Domain theorem implies that $\lambda_n>\lambda_{n-1}$ and $N(\lambda_n)=n-1$.
Inserting this into \eqref{eq:WeylN} (which is specific to Neumann) gives
\begin{equation*}
\lambda_n<\frac{4}{\pi}(n-1)\,.
\end{equation*}
Combining this with \eqref{eq:R4}, we get
\begin{equation*}
n=\mu(\Psi_n) \leq \frac{\bigl|\Omega^{\text{inn}}\bigr|}{\pi {\bf j}^2}\lambda_n + 4\sqrt{\lambda_n}
 < \frac{4}{{\bf j}^2}(n-1)+\frac{8}{\sqrt{\pi}}\sqrt{n-1}\,.
\end{equation*}
A simple calculation shows that this inequality is false if $n\geq 209$.
Hence, in the Neumann case, the analysis of the Courant-sharp situation is reduced to the analysis of the first $208$ eigenvalues.

\section{First reductions}\label{S:hsmall}

\subsection{Analysis via small perturbation}
Here the improvement in comparison with Section \ref{s3}  will result in a better lower bound for the counting function because we are close to the Neumann situation.\\

As $h \rightarrow 0\,$, we  have the following lemma.
 \begin{lem}\label{lem:hsmall}
 There exist $C>0$ and $h_0>0$ such that for $h\in (0,h_0]$ and for each pair $(i,j)$, we have
 \begin{equation}\label{eq:8.1}
 \lambda_{i,j} (0) \leq \lambda_{i,j} (h) \leq  \lambda_{i,j} (0)  + C h\,.
 \end{equation}
 \end{lem}

{\bf Proof.}
 We come back to the computation of $\lambda_{i,j}$. To treat the case  from \eqref{eq:alphaneven},
 we have to analyse the solutions of $\alpha \tan (\frac \alpha 2) = h\pi$.\\
 When $h=0$ the solutions are $\alpha_k = 2 k\pi$ for $k\in \mathbb N$.\\
 If we denote by $\alpha_k(h)$ the solution defined in Section \ref{s4}, we aim to estimate $\delta_k (h) = \alpha_k(h)-\alpha_k$.\\
 First, for $k=0$, we have
 $$
 \delta_0 (h)  \tan \left(\frac {\delta_0(h)}{2}\right) =h \pi\,,
 $$
 which implies the existence of $c>0$ such that for $h$ small,
 $$
 \delta_0 (h) \leq c \sqrt{h}\,.
 $$
 We get immediately
 $$
 \alpha_0 (h)^2 \leq c^2 h\,.
 $$
 We now assume $k>0\,$.\\
 This time we have
 $$
 (2k \pi + \delta_k(h)) \tan  \left(\frac{\delta_k(h)}{2}\right) = h\pi\,.
 $$
 Hence there exists  $h_0 >0$ and $c_0>0$ such that for any $k>0$ and $h\in (0,h_0]$,
 $$
0 < \delta_k (h) \leq c_0 \frac{h}{k}\,.
 $$
 This implies the existence of $h_0 >0$ and $\tilde c >0$ such that for $h\in (0,h_0]$,
 $$
\alpha_k (0) ^2 \leq  \alpha_k (h)^2 \leq \alpha_k (0) ^2  + 2 \alpha_k (0) \delta_k(h) + \delta_k (h)^2 \leq \alpha_k (0) ^2  + \tilde c\, h\,.
$$
The other cases can be treated in a similar way.\\
\qed~\\
 We now have the following lemma.
 \begin{lem} There exists  $h_0 >0$ and $C>0$ such that for any $h\in (0,h_0]$,
  and $\lambda \geq 2$,
$$
 N_{S,h}^R(\lambda) \geq N^R_{S,h=0} (\lambda - C h)= N^{Ne}_S (\lambda- Ch)\,.
 $$
 \end{lem}
Indeed, \eqref{eq:8.1} implies that for each pair $(i, j) \in \N^2$, we have
\begin{equation*}
\lambda_{i,j,h}(S) \leq i^2 + j^2 + Ch,
\end{equation*}
for $h\in (0,h_0]$. So, for each $(i,j)$ satisfying $i^2+j^2 < \lambda - Ch$, we have
$\lambda_{i,j,h}(S) < \lambda$.\\
 Finally we have the following lemma.
\begin{lem}\label{lem:linapprx}
Let $N>0$. Then, there exists  $h_0 >0$ and $C>0$ such that for any $0 <n \leq N$, $h\in (0,h_0]$,
\begin{equation*}
\lambda_{n,h} \leq \lambda_{n,0} + C h\,.
\end{equation*}
\end{lem}
\begin{proof}
We give a proof for fixed $n$. With the notation of \eqref{eq:2.9}, we consider the set $\mathcal I_n$ of the pairs $(i,j)$ such that
$$
i^2 + j^2= \pi^{-2}(\alpha_i (0)^2 + \alpha_j (0)^2) = \lambda_{n,0}\,.
$$
By continuity, for $h$ small enough, we have
$$
\lambda_n (h) \leq \sup_{(i,j) \in \mathcal I_n} \lambda_{i,j,h}\,.
$$
We can then use Lemma \ref{lem:hsmall} to obtain
$$
 \sup_{(i,j) \in \mathcal I_n} \lambda_{i,j,h}\leq \lambda_{n,0} + C h\,.
 $$
\end{proof}

\subsection{Improvement of Theorem \ref{prop:p1} as $h\rightarrow 0$}
If we now come back to the Pleijel-type proof (see Section~\ref{s3}), instead of \eqref{eq:R1} we can use
 \begin{equation*}
  N_{S,h}^{R}(\lambda) \geq N_{S}^{Ne}(\lambda -C h)
  > \frac{\pi}{4}(\lambda -C h)\,.
\end{equation*}
So, assuming that $\lambda$ is Courant-sharp and following the same steps as above, we first get  that
\begin{equation*}
n >  \frac{\pi}{4}(\lambda -C h)+1\, ,
\end{equation*}
instead of \eqref{eq:R2}. Following what is done in Subsection~\ref{ss:3.4},
we obtain that there exists $\tilde C$ such that
\begin{equation*}
n < \frac{4}{{\bf j}^2}(n-1)+\frac{8}{\sqrt{\pi}}\sqrt{n-1} +\tilde C \, h\,.
\end{equation*}
Hence, for $h$ small enough, we get  the following proposition as for the Neumann case.
\begin{prop}\label{prop5.3}
There exists $h_0 >0$ such that, if $h\in [0,h_0]$ and $\lambda_{n,h}$ is an eigenvalue with $n\geq 209$, then it is not Courant-sharp.
\end{prop}

In this way we see that as $h\rightarrow 0$, the number of cases to look at is close to the number obtained for the Neumann case (see Proposition 2.1 of \cite{HPS1}).
It remains to follow, as $h\rightarrow 0$, the other arguments used in \cite{HPS1} to reduce the number of cases to analyse. We do this in Subsection~\ref{ss:decay}, Section~\ref{s5} and Section~\ref{spp}.

\subsection{Multiplicity, labelling and application}\label{ss:decay}
Following the steps of \cite{HPS1}, we now try to eliminate more eigenvalues for $n< 209$, taking advantage of $h$ small enough.
We are now analysing a finite $h$-independent number of cases, and for each case, we will show the existence of $h_0>0$ such that the eigenvalue under consideration is not Courant-sharp for $h\in (0,h_0]$. For the final proof of the main results, we should of course take the infimum of all the $h_0$'s appearing in the case-by-case analysis.

From \eqref{eq:R3} and using that $\mu^{\text{out}}(\Psi)$ is an integer, we have that the analogous result to Lemma 2.2 of \cite{HPS1} holds in the Robin case. \\
Summing \eqref{eq:R3a} over all inner nodal domains gives that
\begin{equation*}
\mu^{\text{inn}}(\Psi) \leq \pi \lambda_{n,h}(S)/\mathbf{j}^2\,,
\end{equation*}
so that
\begin{equation*}
\mu(\Psi)\leq \frac{\pi}{\mathbf{j}^2}\lambda_{n,h}(S) + \mu^{\text{out}}(\Psi).
\end{equation*}
In addition,  with \begin{equation*}
j_n(h) := \sup \{ q: \alpha_{p}(h)^2 + \alpha_{q}(h)^2 = \pi^2 \lambda_{n,h}(S), p, q \in \N\},
\end{equation*}
the analogous result to Lemma 2.3 of \cite{HPS1} holds in the Courant-sharp situation
\begin{equation}\label{eq:4.10}
n=\mu(\Psi) \leq \frac{\pi}{\mathbf{j}^2}\lambda_{n,h}(S) + \max ( 4 j_n(h),1)\,.
\end{equation}
 We wish to compare the right-hand side of \eqref{eq:4.10} to the Neumann situation by using that $h$ is small.
We first recall the main result regarding crossings which was proven in \cite{GH}.
\begin{prop}\label{p:crossing}
For distinct pairs $(p,q)$ and $(p',q')$, with $p\leq q$ and $p'\leq q'$, there is at most one value of $h$ in $[0,+\infty)$ such that
$\lambda_{p,q,h}(S) = \lambda_{p',q',h}(S)$.\\
Moreover,  if $p < p' \leq q' <q$ and $\lambda_{p,q, h^*}=\lambda_{p',q',h^*}$ for some $h^*\geq 0$,
  then \break $h \mapsto \pi^{-2}( \alpha_{p'}(h)^2 + \alpha_{q'} (h)^2 - \alpha_{p}(h)^2 -\alpha_{q}(h)^2)$ is increasing
  for $h>h^*$. In particular  the curve $h\mapsto \pi^{-2}( \alpha_{p}(h)^2 + \alpha_{q} (h)^2)$ is below the curve \break
  $h\mapsto \pi^{-2}( \alpha_{p'}(h)^2 + \alpha_{q'} (h)^2)$ for $h>h^*$.
\end{prop}

We then have the following lemma.
\begin{lem}\label{lem:decay}
The multiplicity of $\lambda_{n,h}$ computed for $h=0$ can only decay as $h$ increases for $h$ small enough.
\end{lem}

\begin{proof}
We first show that for $h$ small enough, curves corresponding to $\lambda_{p,q,h}$ with $\lambda_{p,q,0} \neq
\lambda_{n,0}$ do not intersect curves corresponding to $\lambda_{n,h}$.

Consider $\lambda_{k,h}$ where $\lambda_{k,0}<\lambda_{n,0}$ and $k$ is largest possible.
Then by Lemma~\ref{lem:linapprx}, $\lambda_{k,h} \leq \lambda_{k,0} + Ch < \lambda_{n,0}$ for $h$ small enough.

Similarly, consider $\lambda_{\ell,h}$ where $\lambda_{n,0}<\lambda_{\ell,0}$ and $\ell$ is smallest possible.
Then by Lemma~\ref{lem:linapprx}, $\lambda_{n,h} \leq \lambda_{n,0} + Ch < \lambda_{\ell,0}$ for $h$ small enough.

Hence, we need only consider the curves corresponding to $\lambda_{p,q,h}$ that satisfy $\lambda_{p,q,0}=\lambda_{n,0}$.

It was shown in \cite{HPS1} that for $n \leq 208$, the Neumann eigenvalues of $S$ have multiplicity 1, 2, 3 or 4.

If the multiplicity of $\lambda_{p,q,0}$ is 1 or 2, then it remains constant as $h$ increases (for $h$ small enough). The first case corresponds indeed to $p=q$ and the second case to $p\neq q$.

If the multiplicity of $\lambda_{p,q,0}$ is 3 or 4, then $\lambda_{p,q,0}=\lambda_{p',q',0}=\lambda_{n,0}$ with
$p<p' \leq q'<q$. The first case corresponds indeed to $p'=q'$ and the second case to $p'<q'$. By Proposition~\ref{p:crossing}, the curves corresponding to $(p,q)$ and $(p',q')$ do not intersect
for $h>0$, and  the curve corresponding to $(p,q)$ is below that corresponding
to $(p',q')$. In this case, the multiplicity decreases.
\end{proof}

\begin{rem}\label{rem:decay}
We remark that the above proof also shows that if $\lambda_{n,0}$ has multiplicity 4 such that $\lambda_{p,q,0}= \lambda_{q,p,0} =\lambda_{p',q',0} =\lambda_{q',p',0}$ with $p<p'<q'<q$, then for $h$ small enough, $\lambda_{p,q,h}= \lambda_{q,p,h} < \lambda_{p',q',h} =\lambda_{q',p',h}$.
Similarly, if $\lambda_{n,0}$ has multiplicity 3 such that $\lambda_{p,q,0}= \lambda_{q,p,0} =\lambda_{p',p',0}$ with $p<p'<q$, then for $h$ small enough, $\lambda_{p,q,h}= \lambda_{q,p,h} < \lambda_{p',p',h}$.
\end{rem}

As a consequence of Lemma~\ref{lem:decay}, we get  the following lemma.
\begin{lem}
For any $N>0$, there exists $h_0 >0$, such that if $n\leq N$, $h\in [0,h_0]$ then
$$
j_n(h) \leq j_n(0)\,.
$$
\end{lem}
This immediately leads to  the following lemma.
\begin{lem}\label{lemma2.3}
There exists $h_0 >0$ and $C>0$ such that if $n\leq 208$, $h\in [0,h_0]$, and $\lambda_{n,h}$ is Courant-sharp then
\begin{equation}\label{ineqrob}
n \leq \frac{\pi}{\mathbf{j}^2}\lambda_{n,0}(S) + \max ( 4 j_n(0),1) + C h\,.
\end{equation}
\end{lem}
We can now use the same computations as in  Corollary 2.4 of \cite{HPS1} 
to eliminate, for $h$ small enough, the same cases as for the Neumann problem.

\begin{prop}
\label{prop5.11}There exists $h_0 >0$ such that, if $h\in [0,h_0]$ and $\lambda_{n,h}$ is an eigenvalue where $n$ is one  of 86, 95--96, 99--100, 103--104,
113, 118--119, 120--121, 128--142, 147--208, then it is not Courant-sharp.
\end{prop}
\begin{proof}
The numerical
calculation performed in \cite{HPS1}  shows that
\[
\frac{\pi}{{\bf j}^2}\lambda_{n,0}+  4j_n(0) <n\,,
\]
for the $n$ mentioned in the statement, in contradiction with \eqref{ineqrob}  for $h$ small enough.
\end{proof}

\section{On the use of symmetries}\label{s5}
In this section, we further reduce the potential candidates for Courant-sharp Robin eigenvalues of the square when $h$ is small  by using symmetry properties. This leads us to push the argument due to Leydold further as developed in \cite{HPS1}.

\subsection{Antisymmetric eigenvalues}\label{ss5.1}
 Similarly to \cite{HPS1}, let $L^{\text{ARot}}$ denote the Robin Laplacian restricted to the
antisymmetric space
\[
\mathcal H^{\text{ARot}} = \{\psi~|~\psi(-x,-y)=-\psi(x,y)\}\,.
\]
The spectrum of the Robin Laplacian is given by $\pi^{-2} (\alpha_p(h)^2 + \alpha_q(h)^2)$ with $p+q$ odd.
We denote the sequence of eigenvalues of $L^{\text{ARot}}$ by $(\lambda_{n,h}^{\text{ARot}})_{n=1}^{+\infty}$,
counted with multiplicity. Then each antisymmetric $\lambda_{n,h}$ equals $\lambda_{m,h}^{\text{ARot}}$ for some $m$. The following lemma is an analogue of Courant's Nodal Domain theorem in this subspace, and is proven in \cite{HPS1}
for the Neumann case.
\begin{lem}
\label{lem:antisymmetric}
Assume that $(\lambda_{n,h},\Psi_{n,h})$ is an eigenpair of  the Robin Laplacian with parameter $h$, with $\lambda_{n,h}$ antisymmetric, and let $m$ be such that $\lambda_{n,h}=\lambda_{m,h}^{\text{ARot}}$.
Then $\mu(\Psi_{n,h})$ is even, and
\[
\mu(\Psi_{n,h}) \leq 2 m\,.
\]
\end{lem}

\begin{prop}
\label{cor:antisymmetric}
There exists $h_0 >0$ such that for $h\in (0,h_0]$,
the eigenvalues
$\lambda_{7,h}=\lambda_{8,h}$,
 $\lambda_{23,h}=\lambda_{24,h}$,
$\lambda_{25,h}=\lambda_{26,h}$,
$\lambda_{29,h}=\lambda_{30,h}$,
$\lambda_{36,h}=\lambda_{37,h}$,
$\lambda_{40,h}=\lambda_{41,h}$,
$\lambda_{51,h}=\lambda_{52,h}$,
$\lambda_{55,h}=\lambda_{56,h}$,
 $\lambda_{59,h}=\lambda_{60,h}$, $\lambda_{61,h}=\lambda_{62,h}$,
$\lambda_{72,h}=\lambda_{73,h}$,  $\lambda_{76,h}=\lambda_{77,h}$,
$\lambda_{91,h}=\lambda_{92,h}$,
$\lambda_{97,h}=\lambda_{98,h}$,
$\lambda_{99,h}=\lambda_{100,h}$,
$\lambda_{103,h}=\lambda_{104,h}$,
$\lambda_{109,h}=\lambda_{110,h}$, $\lambda_{111,h}=\lambda_{112,h}$,
$\lambda_{120}=\lambda_{121}$,
$\lambda_{124,h}=\lambda_{125,h}$,
$\lambda_{132,h}=\lambda_{133,h}$,
$\lambda_{143,h}=\lambda_{144,h}$, and $\lambda_{145,h}=\lambda_{146,h}$,
  are not Courant-sharp.
\end{prop}

This was established for $h=0$ in \cite{HPS1} by verifying case by case that $2m < n$. Note that many cases can be more directly obtained by the following lemma.
\begin{lem}\label{lem4.1} Let $0\leq h < +\infty$. Suppose that $\lambda_{n,h}(S)$ is a Robin eigenvalue with corresponding eigenfunction defined in \eqref{eq:2.9}. Suppose that $n$ is odd and that the conditions of Remark~\ref{remsym} are satisfied. Then $\lambda_{n,h}(S)$ is not Courant-sharp.
\end{lem}
The property goes through when the eigenvalue has multiplicity $2$ for $h=0$, see Lemma~\ref{lem:decay}.
We have to be more careful in the case where the multiplicity is higher.\\
Let us first look at the perturbation of $\lambda_{23,0}=\lambda_{24,0}=\lambda_{25,0}=\lambda_{26,0}$.
By Remark~\ref{rem:decay}, for $h$ small enough, we have $\lambda_{23,h}=\lambda_{24,h}<\lambda_{25,h}=\lambda_{26,h}$ corresponding to the pairs $(5,0), (4,3)$ respectively. So Lemma~\ref{lem4.1} shows that the eigenvalue cannot be Courant-sharp.\\
The same argument works for the perturbation of  $\lambda_{59,0}=\lambda_{60,0}=\lambda_{61,h}=\lambda_{62,0}$,  $\lambda_{109,0}=\lambda_{110,0}=\lambda_{111,0}=\lambda_{112,0}$, $\lambda_{143,0}=\lambda_{144,0}=\lambda_{145,0}=\lambda_{146,0}$.\\
Let us finally look at the perturbation of $\lambda_{76,0}=\lambda_{77,0}=\lambda_{78,0}=\lambda_{79,0}$. As proven in \cite{HPS1}, we have $2\times 37 < 76$ where $\lambda_{76,0} = \lambda_{37,0}^{ARot}$. By Proposition~\ref{p:crossing} 
and Remark~\ref{rem:decay}, for $h$ small enough we have $\lambda_{76,h}=\lambda_{77,h} < \lambda_{78,h}=\lambda_{79,h}$. 
We have $\lambda_{76,h}=  \lambda_{m,h}^{ARot}$ and $\lambda_{78,h}=  \lambda_{m+2,h}^{ARot}$, so  $\lambda_{76,h}$ is not Courant-sharp but we cannot conclude for $\lambda_{78,h}$ because $ 2 (m+2) =  78$.
The same problem occurs for the perturbation of $\lambda_{124,0}=\lambda_{125,0}=\lambda_{126,0}=\lambda_{127,0}$.\\

\subsection{Symmetric eigenvalues}\label{ss5.2}

 Similarly to \cite{HPS1}, let $L^{\text{SRot}}$ denote the Robin Laplacian restricted to the symmetric space
\[
\mathcal H^{\text{SRot}} = \{\psi~|~\psi(-x,-y)=\psi(x,y)\}\,.
\]
The spectrum of this Laplacian is given by $\pi^{-2}(\alpha_p(h)^2 + \alpha_q(h)^2)$ with $p+q$ even.
We denote the sequence of eigenvalues of $L^{\text{SRot}}$ by $(\lambda_{m,h}^{\text{SRot}})_{m=1}^{+\infty}$, counted with multiplicity. Each symmetric $\lambda_{n,h}$ equals $\lambda_{m,h}^{\text{SRot}}$ for some $m$.
The following is an analogue of Courant's Nodal Domain theorem in this subspace, and is proven in \cite{HPS1}
for the Neumann case.
\begin{lem}
\label{lem:symmetric}
Let $(\lambda_{n,h},\Psi_{n,h})$ be an eigenpair of the Robin Laplacian with parameter $h$, with symmetric  $\lambda_{n,h}$, and let
$m$ be such that $\lambda_{n,h}=\lambda_{m,h}^{\text{SRot}}$. Then
\[
\mu(\Psi_{n,h}) \leq 2 m\,.
\]
\end{lem}

\begin{prop}
\label{cor:symmetric}
There exists $h_0 >0$ such that for $h\in (0,h_0]$,
the eigenvalues
$\lambda_{27,h}=\lambda_{28,h}$,
 $\lambda_{46,h}=\lambda_{47,h}$,
$\lambda_{63,h}=\lambda_{64,h}$,
$\lambda_{82,h}=\lambda_{83,h}$,
$\lambda_{86,h}$,
 $\lambda_{87,h}=\lambda_{88,h}$,
$\lambda_{107,h}=\lambda_{108,h}$,
$\lambda_{113,,h}$,
 $\lambda_{114,h}=\lambda_{115,h}$,  and
$\lambda_{138,h}=\lambda_{139,h}$
are not Courant-sharp.
\end{prop}
\begin{proof}
As in the previous proposition,  the cases where the multiplicity is 2 follow from what was done in \cite{HPS1} for the Neumann case. We just detail the situation where for $h=0$, the multiplicity is larger than $2$.
In the case of $\lambda_{46,0}=\lambda_{47,0}=\lambda_{48,0}$, we have $m=24$.

For $h>0$ small enough, we have $\lambda_{46,h} = \lambda_{47,h} < \lambda_{48,h}$ by Remark~\ref{rem:decay} so
we  cannot conclude  for $ \lambda_{48,h}$ as $\lambda_{46,0} =\lambda_{22,0}^{SRot}$. This case, will be treated later.\\
In the case,  $\lambda_{87,0}=\lambda_{88,0}=\lambda_{89,0}=\lambda_{90,0}$, we have to consider the situation when
$\lambda_{87,h}=\lambda_{88,h} < \lambda_{89,h}=\lambda_{90,h}$. We know from \cite{HPS1} that
$\lambda_{87,0}= \lambda_{43,0}^{SRot}$ ($m=43$).  Hence we cannot conclude for $\lambda_{89,h}$ for $h >0$.
Finally, in the case  $\lambda_{114,0}=\lambda_{115,0}=\lambda_{116,0}=\lambda_{117,0}$, we know from \cite{HPS1} that
$\lambda_{114,0}= \lambda_{56,0}^{SRot}$ ($m=56$) and we  cannot conclude for $\lambda_{116,h}$ for $h>0$ (Note that $\lambda_{114,h}=\lambda_{115,h}<\lambda_{116,h}=\lambda_{117,h}$).\\
\end{proof}

In comparison with the case $h=0$, we have ``lost'' the treatment of two  eigenvalues: $\lambda_{89,h}$ and  $\lambda_{116,h}$.\\

\subsection{Other symmetries}\label{ss5.3}
Next, similarly to \cite{HPS1}, let $L^{\text{AMir}}$ denote the Robin Laplacian restricted to the
doubly anti-symmetric space
\[
\mathcal H^{\text{AMir}} = \{\psi~|~\psi(-x,y)=-\psi(x,y),\ \psi(x,-y)=-\psi(x,y)\}\,.
\]
The spectrum of this Laplacian is given by $\pi^{-2}(\alpha_p(h)^2 + \alpha_q(h)^2)$ with $p$ and $q$ odd.
We denote the sequence of eigenvalues of $L^{\text{AMir}}$ by $(\lambda_{m,h}^{\text{AMir}})_{m=1}^{+\infty}$, counted with multiplicity. The following lemma is an analogue of Courant's Nodal Domain theorem in this subspace, and is proven in \cite{HPS1} for the Neumann case.

\begin{lem}
\label{lem:antimirror}
Assume that $(\lambda_{n,h},\Psi_{n,h})$ is an eigenpair of  the Robin Laplacian with parameter $h$, with $\lambda_{n,h}$ symmetric
and $\Psi_{n,h} \in \mathcal H^{\text{AMir}} $.
Then
\[
\mu(\Psi_{n,h}) \leq 4 m\,,
\]
for $m$ such that  $\lambda_{n,h}=\lambda_{m,h}^{\text{AMir}}$.
Moreover, $\mu(\Psi_{n,h})$ is divisible by $4$.
\end{lem}

\begin{rem}
If, for all pairs $(p,q)$ of non-negative integers such that
$$\pi^{-2}(\alpha_p(h) ^2+\alpha_q(h)^2)=\lambda_{n,h}\,,$$ it holds that $p$ and $q$ are odd, then there exists an
$m$ such that $\lambda_{n,h}=\lambda_{m,h}^{\text{AMir}}\,$.
\end{rem}

\begin{prop}
\label{cor:antimirror}
The eigenvalues
$\lambda_{12,h}=\lambda_{13,h}$,
$\lambda_{20,h}$,
$\lambda_{27,h}=\lambda_{28,h}$,
$\lambda_{32,h}=\lambda_{33,h}$,
 $\lambda_{46,h}=\lambda_{47,h } < \lambda_{48,h}$
$\lambda_{53,h}=\lambda_{54,h}$,
$\lambda_{68,h}=\lambda_{69,h}$,
$\lambda_{74,h}=\lambda_{75,h}$,
$\lambda_{80,h}=\lambda_{81,h}$,
$\lambda_{86,h}$,
$\lambda_{95,h}=\lambda_{96,h}$,
$\lambda_{107,h}=\lambda_{108,h}$,
 $\lambda_{114,h}=\lambda_{115,h}$,
$\lambda_{128,h}=\lambda_{129,h}$,  and
$\lambda_{140,h}$
are not Courant-sharp.
\end{prop}
\begin{proof}
As for the preceding propositions, the only cases where the situation can change in comparison with $h=0$ are
the cases when the multiplicity is higher than $2$.\\
Looking at  $\lambda_{46,0}=\lambda_{47,0}=\lambda_{48,0}$, we have $\lambda_{46,0}=\lambda_{9,0} ^{\text{AMir}}$. For $h >0$, we have by Proposition~\ref{p:crossing} and  Remark~\ref{rem:decay}, 
$\lambda_{46,h}=\lambda_{47,h} < \lambda_{48,h}$. Then $ \lambda_{48,h}= \lambda_{11,h} ^{\text{AMir}}$ and we are done.\\
Looking at $\lambda_{114,0}=\lambda_{115,0}=\lambda_{116,0}=\lambda_{117,0}$, we have
$\lambda_{114,0}=\lambda_{27,0} ^{\text{AMir}}$. For $h >0$, we get $\lambda_{116,h}=\lambda_{29,0} ^{\text{AMir}}$
and we cannot conclude.
 \end{proof}

Hence, in comparison with $h=0$, we do not know at the moment if the eigenvalue $\lambda_{116,h}$ is Courant-sharp  or not for $h>0$.

\subsection{Reflection in the coordinate axes}
In the case where $h$ is sufficiently small and $p$ and $q$ are even, we consider the analogue of Lemma 3.8 from \cite{HPS1}. We recall that $\Phi_{\theta,p,q}$ was introduced in \eqref{defphitheta}.
\begin{lem}
\label{lem:pandqeven}
Assume that $p$ and $q$ are even and that the equation $\Phi_{\theta,p,q}(\frac \pi 2,y)=0$ has at least $k$ solutions for $-\frac \pi 2<y<\frac \pi 2$ ($k\geq 0$) and the equation $\Phi_{\theta, p,q}^\theta(x,\frac \pi 2)=0$ has at least $\ell$ solutions ($\ell \geq 0$)  for $-\frac \pi 2<x<\frac \pi 2$\,.
Then
\[
\mu(\Phi_{\theta,2p,2q})\leq 4\mu(\Phi_{\theta,p,q})-(2(k+\ell)+3)\,.
\]
\end{lem}

\begin{proof}
The function $\Phi_{\theta, 2p,2q}$ is even in the lines $x=0$ and
$y=0$. For $h$ small enough, we can bound $\mu(\Phi_{\theta, 2p,2q})$ from above by $4\mu(\Phi_{\theta,p,q})$.
We note that for each zero described in the statement (except the biggest one),
we count each nodal domain of $\Phi_{\theta,2p,2q}$  twice.
The nodal domain in the middle is subtracted three times if
$\Phi_{\theta, p,q}(\frac \pi 2,\frac \pi 2)\neq 0$.
\end{proof}
We observe that by Sturm's Theorem (see \cite{BH2,St}) we can take $k=\ell =\min\{p,q\}$  in Lemma~\ref{lem:pandqeven} above. We use this observation below.

\begin{prop}
\label{cor:pandqeven} There exists $h_0 >0$, such that for $h\in[0,h_0)$,
the eigenvalues $\lambda_{38,h}=\lambda_{39,h}$ and $\lambda_{93,h}=\lambda_{94,h}$
are not Courant-sharp.
\end{prop}
\begin{proof}
For $\lambda_{38,0}$ we have $p=2$ and $q=6$ and the labelling is preserved for $h$ small enough.
 Note that $\lambda_{1,3,h=0}$ has the labelling $12$. But we have proven just above  that
$\lambda_{12,h}$ is not Courant-sharp, hence by the above lemma, we get:
 $$
\mu(\Phi_{\theta,2,6} )\leq 4\mu(\Phi_{\theta,1,3} )-(4+3) \leq  4 \times  11 - 7 = 32 \,.
 $$

 For $\lambda_{93,0}$ we have $p=2$ and $q=10$ and the labelling is preserved for $h$ small enough.
 Note that $\lambda_{1,5,0}$ has the labelling $27$. But we have proven just above that
 $\lambda_{27,h}$ is not Courant-sharp, hence by the lemma, we get:
 $$
\mu(\Phi_{\theta, 2,10})\leq 4\mu(\Phi_{\theta,1,5})-(4+3) \leq  4 \times 26 - 7 = 97 \,.
 $$
 This is not enough, but if in addition we use Lemma~\ref{lem:antimirror}, we get $\mu(\Phi_{\theta,1,5})\leq 24$ and
  $$
\mu(\Phi_{\theta,2,10})\leq 4\mu(\Phi_{\theta,1,5})-(4+3) \leq  4 \times 24 - 7 = 89 <  93 \,.
 $$
\end{proof}

In summary, this leaves open the analysis for $h>0$ of the cases $\lambda_{78,h}$, $\lambda_{89,h}$, $\lambda_{116,h}$, and   $\lambda_{126,h}$, which were not Courant-sharp for $h=0$ and for which the previous arguments fail. We treat $\lambda_{116,h}$ $(h>0)$ in Section~\ref{s11}. To deal with the three other cases, in Sections~\ref{s6} and ~\ref{s7}, we develop a similar argument to that of \cite{GH} but for small $ h>0$.

\section{The case $(p,p)$}\label{spp}
 Corresponding to Lemma 4.4 of \cite{HPS1}, we have the following proposition.
\begin{prop}
\label{lem:pp}
 If the eigenspace corresponding to $\lambda_{p,p,0}$ is one-dimensional
then, for $h$ small enough and if $\Psi$ is the eigenfunction associated with $\lambda_{p,p,h}$, we have  $\mu(\Psi)=(p+1)^2$.
\end{prop}
\begin{proof}
The eigenspace is spanned by $\cos (\alpha_p(h) x/\pi) \, \cos (\alpha_p(h)y/\pi) $  for $p$ even, and by
$\sin (\alpha_p(h) x/\pi) \,  \sin (\alpha_p(h)y/\pi) $ for $p$ odd. In each case, this is a
product of a function that depends on $x$ and one that depends on $y$.
For $h$ small enough, each of them has $p$ zeros, and thus the number of nodal domains equals \break  $(p+1)^2$.
\end{proof}
Observing that the eigenvalues $\lambda_{20,0}$,
$\lambda_{31,0}$,
$\lambda_{65,0}$,
$\lambda_{86,0}$,
 and $\lambda_{113,0}$
are simple and correspond to $p=3,4,6, 7$ and $8$, we get:
\begin{cor}
\label{cor:pp}
There exists $h_0 >0$, such that, for $0<h\leq h_0$, the eigenvalues\footnote{We have corrected two misprints concerning the labelling in  Corollary 4.5 in \cite{HPS1}}
$\lambda_{20,h}$,
$\lambda_{31,h}$,
$\lambda_{65,h}$,
$\lambda_{86,h}$ and
$\lambda_{113,h}$,
are not Courant-sharp.
\end{cor}
We note that for $p=5$, the argument does not work for $h=0$ because of the multiplicity $3$, but could be modified by observing that the eigenvalue becomes simple for $h >0$. In any case, this was treated in Proposition \ref{cor:antimirror} by another argument. The cases $p=9$ and $p=10$  were already treated in  Proposition \ref{prop5.11}.\\
We recall that in the case $p=2$, we have Proposition \ref{prop:lambda9}.

\section{Perturbation theory for nodal domains at the boundary, the case $h$ small}\label{s6}

 As in \cite{HPS1,GH}, we make use of a result due to Leydold, \cite{Ley}, that for a $C^\infty$-family of eigenfunctions, the number of nodal domains is constant unless there are stationary points in the interior or the cardinality of the boundary points changes.

In  \cite{GH}, we have shown that the number of nodal domains cannot locally increase around some interior point and at a regular point of the boundary by considering a small perturbation of the Dirichlet case.
The proof made use of the Faber-Krahn inequality, hence cannot be applied for $h$ close to zero for the ``boundary'' nodal domains. Our aim here is to treat a small perturbation of the Neumann case under stronger assumptions, but which
could be generic (or satisfied in many cases).

\subsection{At the boundary far from the corners}
  We consider a $C^\infty$-family of eigenfunctions $\Phi(x,y,\theta,h)$ depending on $\theta$ and $h\in [0,h_0)$,
  with corresponding eigenvalue $\lambda(h)$, satisfying the $h$-Robin condition.
  We assume that on one side  of the square
  (say $x=-\frac \pi 2$ to fix the ideas) there is a point $y_0\in (-\frac \pi 2,\frac \pi 2)$ and some $\theta_0 \in [0,\pi]$ such that the following condition is satisfied, for $z_0=(-\frac \pi 2,y_0)$,
   \begin{equation}\label{6.1a}
  \begin{array}{rl}
   \Phi(z_0,\theta_0,0) & =0\,,\\
   \partial_y   \Phi(z_0,\theta_0,0) & =0\,,\\
    \partial_y^2  \Phi(z_0,\theta_0,0) & \neq 0\,,\\
     \partial_\theta  \Phi(z_0,\theta_0,0) & \neq 0\,.
     \end{array}
     \end{equation}
     We also assume that there exist positive constants  $\epsilon_1$ and $\eta_1$ such that
     in a neighbourhood $B(z_0,\epsilon_1)$ of $z_0$, and for $|\theta-\theta_0| +h \leq \eta_1$ in $\mathbb R^2$, $\Phi(\cdot,\theta,h)$ satisfies;
     \begin{equation}\label{6.2a}
     -\Delta \Phi (x,y,\theta,h) = \lambda (h) \Phi (x,y,\theta,h)\,.
     \end{equation}

     Our aim is to prove the following proposition
     \begin{prop}\label{Proposition6.1} Under Assumptions \eqref{6.1a} and  \eqref{6.2a},
     there exist positive constants $\epsilon_0$ and $\eta_0$ such that in $B(z_0,\epsilon_0)\cap S$ the number of nodal domains of $ \Phi(\cdot ,\theta,h)$
     is $3$ for $(\theta,h) =(\theta_0,0)$ and is $\leq 3$ for $|\theta-\theta_0| + h \leq \eta_0$.
     \end{prop}
     We first give some consequences of \eqref{6.1a}. \\
     We first observe that they imply the following
     \begin{equation*}
     \nabla \Phi (z_0,\theta_0,0) =0\,.
     \end{equation*}
     Hence, for $(\theta,h)=(\theta_0,0)$, $z_0$ is a critical point of $\Phi (\cdot,\theta_0,0)$ considered as a function on $\mathbb R^2$.\\
     The second consequence is that  by differentiating the Neumann condition  tangentially, we obtain that
     \begin{equation*}
     \partial^2_{x,y} \Phi (z_0,\theta_0,0) =0\,.
     \end{equation*}
     Coming back to our assumptions, we can assume (w.l.o.g) that
     \begin{equation}\label{6.1b}
     \begin{array}{rl}
   \Phi(z_0,\theta_0,0) & =0\,,\\
   \partial_y   \Phi(z_0,\theta_0,0) & =0\,,\\
    \partial_{yy}^2  \Phi(z_0,\theta_0,0) & < 0\,,\\
     \partial_\theta  \Phi(z_0,\theta_0,0) & < 0\,.
     \end{array}
     \end{equation}
     We can indeed always come back to this situation by replacing $\Phi$ by $-\Phi$ or $\theta$ by $-\theta$.\\
     The third consequence of \eqref{6.1b} and \eqref{6.2a} is that
     \begin{equation} \label{6.2b}
       \partial_{xx}^2  \Phi(z_0,\theta_0,0)  > 0\,.
     \end{equation}

     Hence, $z_0$ is a zero critical point of $\Phi (\cdot,\theta_0,0)$ with non-degenerate Hessian with signature $(-,+)$.
      It then follows, using \eqref{6.2a}, the local structure of an eigenfunction of the Laplacian in the neighbourhood of $z_0$,  and the Neumann condition, that there exist  positive constants $\epsilon_0 <  \epsilon_1$ and $\eta_0<  \eta_1$ such that in $B(z_0,\epsilon_0)\cap S$ the  nodal set of $  \Phi(\cdot,\theta_0,0)$
      consists  of two half-lines emanating from $z_0$  separated by angle $\frac \pi 2$ and making angle $\frac \pi 4$ with $x=-\frac \pi 2$ and crossing $\partial B(z_0,\epsilon_0) \cap \bar S$ transversally at exactly two points $z_1$ and $z_2$ in $S$, see Figure~\ref{fig:Nbdry}.

      \begin{figure}[htp]
      \centering
      \includegraphics[width=3cm]{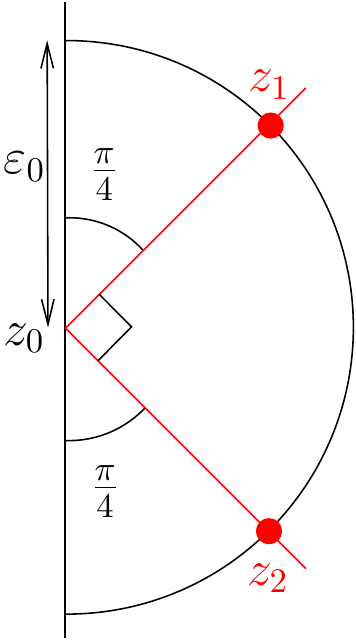}
      \caption{Two nodal lines emanating from $z_0$ separated by angle $\frac{\pi}2$.}
      \label{fig:Nbdry}
      \end{figure}

     The fourth consequence (that follows from the first and last lines in \eqref{6.1b}) is that there exists $\theta_\pm$ such that
     \begin{equation}\label{6.3b}
     \begin{array}{l}
     \theta_- < \theta_0 < \theta_+\,,\\
      \Phi
     (z_0,\theta_-,0) >0 >\Phi
     (z_0,\theta_+,0)\,,\\
     \partial_\theta \Phi
     (z_0,\theta, 0) < 0\,,\, \mbox{ for } \theta \in [\theta_-,\theta_+]\,.
     \end{array}
     \end{equation}
      We note that these properties are stable. There exist positive (possibly smaller) $\epsilon_0$ and $\eta_0$ such that, for $|\theta-\theta_0|+h \leq \eta_0$ the zero set
       of $\Phi (\cdot,\theta,h)$ is crossing $\partial B(z_0,\epsilon_0) \cap \bar S$ transversally at exactly two points $z_1(\theta,h)$ and $z_2(\theta,h)$ in $S$.
       Moreover, we have
       \begin{equation*}
       \lim_{(\theta,h)\rightarrow (\theta_0,0)} z_j(\theta,h) = z_j \mbox{ for } j=1,2\,.
       \end{equation*}
        We now apply standard Morse theory for
       $$
       \psi (y,\theta,h) := \Phi (-\frac \pi 2,y,\theta,h)\,.
       $$
       This is a Morse function for $(\theta,h)$ sufficiently close to $(\theta_0,0)$ and $\psi(\cdot,\theta,h)$ admits a unique critical point $y(\theta,h)$ close to $y_0$:
       \begin{equation*}
       \partial_y \psi (y(\theta,h),\theta,h)=0\,.
       \end{equation*}
       We now look at the behaviour of  the function $\chi$ which is defined by
       \begin{equation*}
       \chi (\theta,h) :=\psi (y(\theta,h), \theta,h)\,.
       \end{equation*}
      We note that $\partial_\theta \chi$ is close to $\partial_\theta \psi$ and using the stability of \eqref{6.3b}, we obtain the existence of a (possibly smaller) $ \eta_0 > 0$ such  that for $0 \leq h\leq \eta_0\,$, there exists
      a unique $ \theta (h) \in (\theta_0-\eta_0,\theta_0 +\eta_0)$ such that
       \begin{equation*}
       \chi (\theta(h), h)= 0 \mbox{ and } \theta(0)=\theta_0\,.
       \end{equation*}
       We now observe that,  by construction, $y (\theta(h),h)$ satisfies
       \begin{equation*}
       \partial_y \psi (y(\theta(h)),h),\theta(h) ,h)= \psi (y(\theta(h),h),\theta(h),h) =0\,.
       \end{equation*}
       We would like to deduce that for $\theta< \theta(h)$, $\psi (y,\theta,h)$ has two zeros $y_\pm(\theta,h)$ such that
       $y_- (\theta,h) < y(\theta,h) < y_+ (\theta,h)$ and that, for $\theta >\theta(h)$, $\psi (y,\theta,h)$ has no zero.\\
       Using the Taylor expansion with integral remainder term, we can write
        $$
         \psi (y ,\theta ,h)- \chi (\theta,h) = c(y,\theta,h)  (y-y(\theta,h))^2\,.
         $$
         where $c(y_0,\theta_0,0) < 0$.\\
         We make the change of variable:
         $$
         y \mapsto \tilde y = \sqrt{-c (y,\theta,h)} (y-y(\theta,h))
         $$
         and denote by $\tilde y \mapsto y =\zeta( \tilde y , \theta,h)$ its inverse map.\\
         If $y$ is a zero of $\psi$, we obtain:
         $$
         \chi (\theta,h) = \tilde y ^2\,,
         $$
         Hence, if $\chi (\theta,h)\geq 0$,  we have two solutions
         $$
         \tilde y_\pm =\pm \sqrt{\chi(\theta,h)}\,.
         $$
        Coming back to the initial coordinates, we get
       $$
       y_\pm (\theta,h) = y(\theta,h) \pm \frac{\sqrt{\chi (\theta,h)}}{ \sqrt{-c ( \zeta (\tilde y_\pm,\theta,h),\theta,h)}    } \,.
       $$
Using the Robin condition, we also obtain that
         \begin{align*}
       \partial_x \Phi (-\pi/2, y(\theta(h),h),\theta(h),h)&=  h \Phi (-\pi/2, y(\theta(h),h),\theta(h),h) \notag\\
       & = h  (\psi (y(\theta(h),h),\theta(h),h) =0\,.
       \end{align*}
       Hence, for $h$ small enough,
       $$z_c(h):=(-\pi/2, y (\theta(h),h))$$
        is a zero critical point of $\Phi (\cdot,\theta(h),h)$.\\

        At this stage, we have achieved the analysis of the intersection of the zero set at the boundary where we have distinguished three situations depending on $\theta-\theta(h)$:  the intersection consists of two points, one double point and no point in a fixed  neighbourhood of $z_0$ in the boundary $(y_0-\epsilon_0,y_0 + \epsilon_0)$.\\

       In order to control the topology of the nodal set in $B(z_0,\epsilon_0)\cap \bar S$, we now analyse if $\Phi (\cdot, \theta,h)$ can have critical points in $S\cap B(z_0,\epsilon_0)$.\\
       As a function of $(x,y)$, $\Phi(\cdot,\theta,h)$ is a Morse function for $(\theta,h)$ close to $(\theta_0,0)$. Hence, there exists a unique critical point $\hat z (\theta,h)$ close to $z_0$. Let us show that if it is a zero critical point it has to be on the boundary. If there was such a zero critical point, we would have, using the Taylor expansion of $\Phi (z, \theta,h)$ at $\hat z (\theta,h)$ and taking
       $z= \check z (\theta,h):= (-\frac \pi 2, y (\theta,h))\,,$
     $$
     \begin{array}{ll}
     0&= \Phi (\hat z (\theta,h), \theta,h)  \\
     & = \Phi ( \check z (\theta,h), \theta,h)  +  \mathcal O (|\hat z (\theta,h)-\check z(\theta,h)|^2)\,.
    \end{array}
     $$
We now observe that $$\check z(\theta(h),h) =  \hat z (\theta(h),h)=z_c(h)$$
by uniqueness of the critical point, hence
  \begin{equation*}
  \hat z (\theta,h)-\check z(\theta,h)= \mathcal O (|\theta -\theta(h)|)\,.
  \end{equation*}
   This implies
  \begin{equation*}
  0 =  \Phi ( \check z (\theta,h), \theta,h) + \mathcal O (|\theta-\theta(h)|^2) \,.
   \end{equation*}
   We now take the Taylor expansion of $\Phi ( \check z (\theta,h), \theta,h)$ at $\theta=\theta(h)$ to obtain
   $$
   \Phi ( \check z (\theta,h), \theta,h) = \partial_\theta \Phi ( \check z (\theta(h),h), \theta(h),h)(\theta-\theta(h)) + \mathcal O (|\theta-\theta(h)|^2)\,.
   $$
    Altogether, we have
   $$
   \partial_\theta \Phi ( \check z (\theta(h),h), \theta(h),h)(\theta-\theta(h)) = \mathcal O (|\theta-\theta(h)^2)
   $$
   whose unique solution is $\theta=\theta(h)$, observing that $  \partial_\theta \Phi ( \check z (\theta(h),h), \theta(h),h)\neq 0$.  We conclude that $\check z(\theta(h),h) =  \hat z (\theta(h),h)=z_c(h) \in \partial S$
   is the only possibility.\\

   We now look at the topology of the nodal set in $\bar S \cap \bar B (z_0,\epsilon_0)$.
   If $\theta > \theta(h)$, we have no point at the boundary and the only possibility (having in mind that we have no zero critical point) is an arc inside  $S$ joining $z_1(\theta,h)$ and $z_2(h,\theta)$. In this case, we have two nodal sets in $\bar S \cap \bar B (z_0,\epsilon_0)$.\\
   If $\theta=\theta(h)$ the only possibility is that the nodal set consists of two arcs joining $z_j (\theta,h)$ ($j=1,2$) and $\check z (\theta(h),h)$ at the boundary.
   In this case, we have three nodal sets.\\
   Finally if $\theta < \theta(h)$, we can exclude the possibility that there is one arc joining $y_-(\theta,h)$ and  $y_+(\theta,h)$ and another joining  $z_1(\theta,h)$ and  $z_2(h,\theta)$. Of course they cannot intersect because this would create a critical point.
   We observe that
     $$\Phi
     (\check z (\theta,h) ,\theta,h) >0\,.
     $$
     Now we have for $x\geq -\frac \pi 2$, taking the Taylor expansion  of
     order $2$ at the point  $-\frac \pi 2$ in the first variable,
     $$\Phi
      (x, y(\theta,h) ,\theta,h) = \Phi (\check z (\theta,h) ,\theta,h)+ \partial_x \Phi (\check z (\theta,h) ,\theta,h)(x+\frac \pi 2) + d(x,\theta,h)  (x+\frac \pi 2) ^2\,,
     $$
     with $d(-\frac \pi 2,\theta_0,0) = \partial_{xx}\Phi (-\frac \pi 2, y_0,\theta_0,0) >0$ (by \eqref{6.2b}).

     Using again the Robin condition and the non-negativity of $h$, we observe that
     $$
     \partial_x \Phi (\check z (\theta,h) ,\theta,h) = h \Phi (\check z (\theta,h) ,\theta,h) \geq 0\,.
     $$
     Hence the nodal set for $h\geq 0$ does not meet the horizontal segment $y=y(\theta,h)$ in $B(z_0,\epsilon_0)$. The only remaining possibility is consequently that the nodal set consists of two  non-intersecting arcs
     connecting two points of the boundary to two points of $\partial B (z_0,\epsilon_0)\cap S$ determining three nodal sets.\\
     This achieves the proof of the proposition.\qed \\

  It remains to understand what can occur at the corners.

\subsection{Analysis at the corner}
    We begin with the Neumann case and we work in $\hat S:=(0,\pi)^2$
    because on $\hat{S}$ there is a unique expression for the eigenfunctions (this allows us to avoid a discussion of four different cases below).\\
    We are interested in the families (for say $p<q$)
    $$
    \Phi(x,y,\theta,0) =\cos \theta \cos p x  \cos qy + \sin \theta \cos py \cos qx\,.
    $$
    We look at the situation at the corner $(0,0)$. The first question is to know if there is a zero. We observe that
    $$
    \Phi (0,0,\theta)= (\cos \theta + \sin \theta).
    $$
    Hence the only case is for $\theta_0=\frac {3\pi}{4}$. \\
    If we look at other corners we only meet the same $\theta$ or $\theta= \frac \pi 4$ depending on the parities of $p$ and $q$.
   Due to the Neumann condition, the corner $z_0=(0,0)$ is a critical point. Looking at the second derivative, we observe that
   $$
    \partial_{xx}^2  \Phi(0,0,\theta_0,0) = \cos \theta_0 (-p^2) + \sin \theta_0 (-q^2)= \cos \theta_0 (q^2-p^2) <  0\,.
    $$
    Similarly
     $$
    \partial_{yy}^2 \Phi(0,0,\theta_0,0) = - \cos \theta_0 (q^2-p^2) > 0\,.
    $$
    with opposite sign, and
      $$
    \partial_{xy}^2  \Phi(0,0,\theta_0,0) = 0\,.
    $$
    The zero set for $\theta=\theta_0$  is simply $x=y\,$.\\
    Finally, we note that
    $$
    \partial_\theta  \Phi (0,0,\theta_0) = 2 \cos \theta_0 <  0\,.
    $$
    The guess is simply that this situation is stable for $(\theta,h)$ close to $(\theta_0,0)$.\\

    Coming back to $S$ and supposing that $p$ and $q$ are even to fix the ideas ($p=2 \hat p$ and $q= 2 \hat q$), we
    focus on $z_0= (-\frac \pi 2,-\frac \pi 2)$. The family we are interested in is
    $$
    \Phi_{p,q} (x,y,\theta,h):= \cos \theta \cos  (\alpha_p  x/\pi) \cos(\alpha_q   y/\pi) + \sin \theta  \cos  (\alpha_p  y /\pi) \cos(\alpha_q   x/\pi)
    $$
    and we have
    $$
    \Phi_{p,q} (-\pi/2,-\pi/2,\theta,h):= ( \cos \theta + \sin \theta) \cos  (\alpha_p /2) \cos(\alpha_q  /2)\,,
    $$
    with
    $$
    \cos  (\alpha_p /2) \cos(\alpha_q  /2) = (-1)^{\hat{p} + \hat{q}} + \mathcal O (h) \neq 0\,.
    $$
    The other cases are similar.\\

   Hence, we consider a $C^\infty$-family of eigenfunctions $\Phi(x,y,\theta,h)$ depending on $\theta$ and $h\in [0,h_0)$,  with corresponding eigenvalue $\lambda(h)$, satisfying the $h$-Robin condition
   and assume that at a corner $z_0 = (x_0,x_0)$, we have for some $\theta_0$ and $h_0>0$
   \begin{equation}\label{6.1c}
  \begin{array}{rl}
   \Phi(x,x,\theta_0,h) & =0\,,\,\forall h\in [0,h_0]\\
   \nabla   \Phi(z_0,\theta_0,h) & =0\,,\,\forall h\in [0,h_0]\\
    \partial_y^2  \Phi(z_0,\theta_0,0) & \neq 0\,,\\
     \partial_\theta  \Phi(z_0,\theta_0,0) & \neq 0\,.
     \end{array}
     \end{equation}
     We also assume that there exist positive constants $\epsilon_1$ and $\eta_1$ such that
     in a neighbourhood $B(z_0,\epsilon_1)$ of $z_0$ and for $|\theta-\theta_0| +h \leq \eta_1$ in $\mathbb R^2$, $\Phi(\cdot,\theta,h)$ satisfies;
     \begin{equation}\label{6.2c}
     -\Delta \Phi (x,y,\theta,h) = \lambda (h) \Phi (x,y,\theta,h)\,.
     \end{equation}
     If we are at a corner of the form $(x_0,-x_0)$, we are in a similar situation with the diagonal $ y=x$ replaced by the diagonal $ y=-x$.

     Our aim is to prove the following proposition
     \begin{prop}\label{Proposition6.2} Under Assumptions  \eqref{6.1c} and \eqref{6.2c},
     there exist positive constants $\epsilon_0$ and $\eta_0$ such that in $B(z_0,\epsilon_0)\cap S$ the number of nodal domains of $ \Phi(\cdot ,\theta,h)$
     is $2$ for $|\theta-\theta_0| + h \leq \eta_0$.
     \end{prop}

     The proof is rather close to that of the previous subsection, except that we have the additional information that $\theta (h)=\theta_0$ for each side. The intersection of the nodal set with the boundary is a point $z(\theta,h)$ which moves from one side to the next side. By assumption $z(\theta_0,h)$ is the unique critical point of $\Phi (\cdot,\theta_0,h)$
     and $z(\theta_0,h)=z_0$. \\

   The second step is to show that the only zero critical point is for $\theta =\theta_0$ and at the corner.
   We suppose that $\hat z (\theta,h)$ is a zero critical point of $\Phi (\cdot,\theta,h)$ and as in the previous subsection we get
   $$
  0 = \Phi ( \hat z (\theta,h),\theta,h)= \Phi (z(\theta,h) , \theta,h) + \mathcal O ( | \hat z (\theta,h)- z(\theta,h) |^2)\,.
  $$
  By uniqueness of the critical point, we have
  $$ \hat z (\theta,h)- z(\theta,h)=  \mathcal O (|\theta-\theta_0|)\,.
  $$
  On the other hand we have
    $$
  \Phi ( z(\theta,h) , \theta,h)  = \partial_\theta \Phi  (z_0,\theta_0, h) (\theta-\theta_0)
   + \mathcal O (|\theta - \theta_0|)^2)\,.
   $$
    Using $\partial_\theta  \Phi  (z_0,\theta_0, 0) \neq 0$, we get $\theta=\theta_0$ and that the critical point should be for $\theta_0$ and equal to $z_0$.

   The last step is easier, we can only have a curve joining the corner and the unique point on $\partial B (z_0,\epsilon_0) \cap S$ which is actually quite close to the diagonal.

\section{Zero critical points at the boundary}\label{s7}

Note that in this section, the eigenfunctions are written on $\hat S:=(0,\pi)^2$. The advantage is that we have a unique expression for the eigenfunctions in the case $h=0$.  Since the case where $p=q$ was already treated in Section~\ref{spp},
we assume that $p \neq q$ in what follows.

\subsection{Analysis at the corner}
We are only interested in families of eigenfunctions for the Neumann case in the form
\[
\Phi_{p,q}(x,y,\theta)
=\cos \theta \cos p x \cos q y + \sin \theta \cos p y \cos q x\,.
\]
It is not difficult to show that the corners $(0,0)$ and $(\pi,\pi)$ only belong to the nodal set
when $\cos \theta + \sin \theta =0$, hence  for
$\theta= \frac{3\pi}{4}$.
For the corners $(0,\pi)$ and $(\pi, 0)$, we get $\cos \theta + \sin \theta (-1)^{p+q} =0$ which leads to $\theta =\frac \pi 4$ or $\frac{3\pi}{4}$ depending on the parity of $p+q$.\\
In  each
case, we observe that the diagonal arriving  at
the corner belongs to the zero set.
For $h >0$, we observe that the same corners are zeros and critical for the same value of $\theta \in \{\frac \pi 4,\frac{3\pi}{4}\}$.\\

 In order to apply Proposition~\ref{Proposition6.2}, we have that equation \eqref{6.1c} holds since we consider the Robin eigenfunctions and the first two conditions of \eqref{6.2c} are satisfied, as observed above. It remains to verify whether
the third and fourth conditions of \eqref{6.2c} are satisfied. We discuss this at the end of  Subsection~\ref{ss:8.2}.

\subsection{Analysis at the boundary outside the corners}\label{ss:8.2}

The zero critical points  of $z \mapsto \Phi_{p,q}(z,\theta)$ are determined   by
\[
\begin{aligned}
\cos \theta \cos p x \cos q y + \sin \theta \cos p y \cos q x &=0\,,\\
p  \cos \theta  \sin p x \cos q y + q \sin \theta \cos p y \sin q x &=0\,,\quad\text{and}\\
q \cos \theta \cos p x \sin q y + p \sin \theta \sin p y \cos q x &=0\,.
\end{aligned}
\]
On $\hat{S}=(0,\pi)^2$, by the analogous results to \eqref{eq:2.8} and \eqref{eq:thetared1},
it suffices to consider the boundary edge $x=0$.
If we are on the side $x=0$, outside the corner,  then, for $y\in (0,\pi)$,  we get
\begin{equation*}
\begin{aligned}
\cos \theta \cos q y + \sin \theta \cos p y &=0\,,\\
q \cos \theta \sin q y + p \sin \theta \sin p y  &=0\,.
\end{aligned}
\end{equation*}
\begin{rem}\label{rem:8.1}
We note that, for a zero critical point $(0,y)$ at the boundary $\{x=0\}$, we have necessarily $\theta \neq 0$ and $ \theta \neq \frac \pi 2$ and that  $\cos py =0$ if and only if  $\cos q y =0$.
\end{rem}

If
 \begin{equation}\label{eq:7.2}
 \cos py \, \cos q y \neq 0\,,
 \end{equation}
  we get
\begin{equation*}
\begin{aligned}
\tan \theta &=   -\cos py/\cos qy \,,\\
q  \tan  qy & = p \tan  py  \,.
\end{aligned}
\end{equation*}

We first consider the particular cases $(0,q)$ or $(p,0)$ (with $p>1$ or $q>1$). In this case the above condition  \eqref{eq:7.2} is satisfied and we get
\begin{lem}\label{lem8.2}~
\begin{enumerate}
\item If $p=0$ and $q>1$, we have $y = k \frac{\pi}{q}$ for some $k \in \{1,\dots, q-1\}$ and $\tan \theta = (-1)^{k+1}$.
\item If $p >1$ and $q=0$, we have  $y = k \frac{\pi}{p}$ for some $k \in \{1,\dots, p-1\}$ and $\tan \theta = (-1)^{k+1}$.
\end{enumerate}
Moreover this critical point is non-degenerate and satisfies the assumptions of Proposition \ref{Proposition6.1}.
\end{lem}

\begin{proof}
 It is not difficult to check that the critical point is non-degenerate so we only detail the last statement. Note that we have $\theta \in \{\frac \pi 4,\frac{3\pi}{4}\}$.

In the first case, we just observe that at the critical point $(0,  k \frac{\pi}{q})$, we have
$$
|\partial_{yy}^2 \Phi_{0,q}(0,  k \frac{\pi}{q})| = q^2 /\sqrt{2} \neq 0\,.
$$
The other condition reads
$$
\partial_{\theta} \Phi_{0,q}(0,  k \frac{\pi}{q})= \sin \theta (-1)^{k+1} + \cos \theta = 2 \sin \theta (-1)^{k} \neq 0 \,.
$$
 (Note that $k$ even corresponds to $\theta=\frac{3\pi}{4}$ and $k$ odd corresponds to $\theta=\frac{\pi}{4}\,$.)
\end{proof}

We now consider the case when $pq\neq 0\,$, $p\neq q\,$.\\
\paragraph{Analysis of \eqref{eq:7.2}}~\\
We first assume that $p$ and $q$ are mutually  prime.\\
Let us consider a solution of $\cos py =0$ and $\cos qy =0$ with $y\in (0,\pi)$.\\
From the assumptions we get that
\begin{equation} \label{eq:7.4}
y = \frac{(2m+1)\pi}{2 p}  = \frac{(2n+1)\pi}{2 q}\,,
\end{equation}
where $m$ and $n$ are non-negative integers, satisfying
\begin{equation}\label{eq:7.5}
(2m+1) < 2p \mbox{ and } (2 n +1) < 2q\,.
\end{equation}
\eqref{eq:7.4} implies
\begin{equation*}
(2m+1) q = (2n+1) p\,.
\end{equation*}
 Since $p$ and $q$  are mutually prime, this implies the existence of a positive integer $\ell$ such that
$$
(2m+1) = \ell p\,.
$$
This would imply, by \eqref{eq:7.5}, $\ell=1$ and $p= (2m+1)$ and $q=(2n+1)$ and $p+q$ even.\\
Hence neither $\cos py =0$ nor $\cos qy =0$ can
occur if $p+q$ is odd (see Remark~\ref{rem:8.1}).\\
\begin{lem}\label{lem8.3}
If $p$ and $q$ are mutually prime and $p+q$ is odd, then we can apply Proposition \ref{Proposition6.1}
 on the boundary $x = 0$ away from the corners.
\end{lem}
\begin{proof}
  Because $\cos qy \cos py \neq 0$, the critical points  on
  the boundary $ x=0$ are critical points of the function
 $$
 y \mapsto f_{p,q} (y) = \frac{\cos py}{\cos qy}\,.
 $$
  We do not need to discuss the localisation or the existence of these zero critical points, but  we have only
  to verify if the assumptions of Proposition~\ref{Proposition6.1} are satisfied  at such a point.

The Hessian is given by
 $$
 {\rm Hessian} ( \Phi_{p,q}) =\left(
 \begin{array}{cc}
 \cos py  \sin \theta (p^2-q^2) & 0 \\
 0 & \cos py  \sin \theta (q^2-p^2)
 \end{array}
 \right)
 $$
Hence this Hessian is non-degenerate because $p\neq q$ and $\theta\neq 0, \frac \pi 2$.\\
We also have that
$$
\partial_{\theta}  \Phi_{p,q}(0,y) = -\sin \theta \cos qy + \cos \theta \cos py\, ,
$$
and, since $(0,y)$ belongs to the nodal set, that
$$ \cos qy = -\frac{\sin \theta \cos py}{\cos \theta}\,.$$
Hence we obtain
$$
\partial_{\theta}  \Phi_{p,q}(0,y)  = \frac{\cos py}{\cos \theta} \neq 0\,.
$$
So Proposition~\ref{Proposition6.1} can be applied when $p,q$ are mutually prime and $p+q$ is odd.
\end{proof}

\begin{rem}\label{rem:8.4}
Suppose $p$ and $q$ are mutually prime, $p+q$ is even and $(0,y)$ is a critical point.
If $\cos py \neq 0$, then, by Remark~\ref{rem:8.1}, $\cos qy \neq 0$ and so
Proposition~\ref{Proposition6.1} applies (as in the proof of Lemma~\ref{lem8.3}).

If $\cos py =0\,$, then, by Remark~\ref{rem:8.1}, $\cos qy = 0$.
As above we have $p=2m+1$, $q=2n+1$ and $y = \frac{\pi}2$.
We have that $\partial_{\theta} \Phi_{p,q}(0,\frac{\pi}2) = 0$ if and only if
$$\tan \theta = \frac{q}{p} (-1)^{m+n-1}.$$
Hence for $\theta \neq \arctan(\frac{q}{p} (-1)^{m+n-1})$, Proposition~\ref{Proposition6.1} can be applied.
The remaining case $\theta = \arctan(\frac{q}{p} (-1)^{m+n-1})$ must be checked.

\end{rem}

 When $p$ and $q$ are not prime, we will see how we can reduce to this situation by looking at sub-squares.
 Indeed, in order to check whether conditions \eqref{6.1a} and the third and fourth conditions of \eqref{6.2c} are satisfied, we need only consider $h=0$ for which the folding argument used in \cite{HPS1} holds. More precisely,
 suppose that $p=k \tilde p$ and $q= k \tilde q$ with $\tilde p$ and $\tilde q$ mutually prime and $\tilde p + \tilde q$ odd.\\
 dividing $(0,\pi)^2$  into $k^2$  sub-squares, we can investigate whether
 the corners of these sub-squares lying in $x=0$ are in the zero set of $\Phi_{p,q}(\cdot,\theta)$  or not.
 We have indeed at a point $(0,  \frac{\ell \pi}{k})$,
 \[
\Phi_{p,q}(0 ,  \ell \pi /k ,\theta)
=\cos \theta \cos \ell \tilde q \pi  + \sin \theta \cos \ell  \tilde p \pi = ( \cos \theta -\sin \theta) (-1)^{\ell \tilde q}\,.
\]
So we have a special case (that the point belongs to the nodal set) when $\theta =\frac \pi 4$.\\

We now perform the change of variable $y \mapsto \frac 1 k  \tilde y + \frac{\ell \pi}{k}$. In this new coordinate, we have, for $y\in ( \frac{\ell \pi}{k}, \frac{(\ell+1)\pi}{k})$
$$
\begin{array}{ll}
\Phi_{p,q}(0 , y) & = \cos \theta  \cos (\ell \tilde q \pi + \tilde q y) + \sin \theta  \cos (\ell \tilde p \pi + \tilde p y) \\
& =  (-1)^{ \ell \tilde q} ( \cos \theta  \cos ( \tilde q y) - \sin \theta  \cos ( \tilde p y))\,.
\end{array}
$$
We are again as in the previous case in the new variables $(\tilde x,\tilde y)$ except that the left corners of the sub-squares (which could be zeros when $\theta =\frac \pi 4$)
 are interior points of the boundary of the initial square. But the analysis there is not difficult  and has already been discussed above. Hence we can generalise the previous lemma in the following way:
 \begin{lem}\label{lem8.5}
If $p =k \tilde p$ and $q=k \tilde q$ with $k\geq 1$, $\tilde p$ and $\tilde q$ mutually prime and $\tilde p+\tilde q$  odd, then we can apply Proposition \ref{Proposition6.1}  on the boundary $x=0\,$.
\end{lem}

\begin{rem}
We remark that by the above, if $p$ and $q$ are mutually prime with $p+q$ odd, then the third and fourth
conditions of \eqref{6.2c} are satisfied, and hence Proposition~\ref{Proposition6.2} can be applied.
If $p$ and $q$ are mutually prime with $p+q$ even, then the case $\theta = \arctan(\frac{q}{p} (-1)^{m+n-1}$
has to be checked by another approach, otherwise Proposition~\ref{Proposition6.2} can be applied.
\end{rem}

\section{Application to  non Courant-sharp situations}\label{s9}
It is immediate to see that, due to the monotonicity of the Robin eigenvalues with respect to $h$, the labelling of an eigenvalue  corresponding to the pair $(p,q)$
 can only increase when going from $h=0$ to $h>0\,$. \\
Hence, starting with a Neumann eigenvalue that is not Courant-sharp, in order to show that the corresponding Robin eigenvalue with $h$ sufficiently small is not Courant-sharp,
it is enough to show that the number of nodal domains does not increase  as $h$ (small) increases. \\
We now list the remaining cases, where the arguments of the previous subsections apply.

\subsection{Treatment of the special cases in \cite{HPS1} for $h>0$ small}
In \cite{HPS1}, there were certain Neumann eigenvalues for which the Courant-sharp property could not be determined
by the Pleijel-inspired strategy or by symmetry arguments. To show that these eigenvalues are not Courant-sharp, a more in-depth analysis was required.
Below we consider these special cases and apply the results of Section~\ref{s7}, together with the result for $h=0$ of \cite{HPS1}, to show that the corresponding Robin eigenvalues are not Courant-sharp for $h$ small enough.

\begin{enumerate}
\item The case $\lambda_7=\lambda_8=5$ ($(p,q)=(2,1)$) \\
Because $p$ and $q$ are mutually prime and $p+q$ is odd, the general theory of Section~\ref{s7} applies.
\item The case $\lambda_{21}=\lambda_{22}=20$ ($(p,q)=(4,2)$)\\
By dividing into four sub-squares,  we obtain four copies of the previous situation.
So the results of Section~\ref{s7} apply.
\item  The case $\lambda_{70}=\lambda_{71}=80$ ($(p,q)=(8,4)$)\\
 Similarly, by dividing into sixteen sub-squares, we come back to the situation of item 1.
\item The case $\lambda_{42}=\lambda_{43}=45$ ($(p,q)=(6,3)$)\\
By dividing into nine sub-squares, we come back to the situation of item~1.
\item The case $\lambda_{14}=\lambda_{15}=13$ ($(p,q)=(3,2)$)\\
Because $p$ and $q$ are mutually prime and $(p+q)$ is odd, the general theory of Section~\ref{s7} applies.
\item The case $\lambda_{49}=\lambda_{50}=52$ ($(p,q)=(6,4)$)\\
We can come back to the previous situation by dividing into four sub-squares.
\item The case $\lambda_{18}=\lambda_{19}=17$ ($(p,q)=(4,1))$\\
Because $p$ and $q$ are mutually prime and $p+q$ is odd, the general theory of Section~\ref{s7} applies.
\item The case $\lambda_{66}=\lambda_{67}= 73$ ($(p,q)= (8,3)$)\\
Because $p$ and $q$ are mutually prime and $p+q$ is odd, the general theory of Section~\ref{s7} applies.
\item The case $\lambda_{84}=\lambda_{85}= 97$   ($(p,q)= (9,4)$).\\
Because $p$ and $q$ are mutually prime and $p+q$ is odd, the general theory of Section~\ref{s7} applies.
\item The case $\lambda_{101}=\lambda_{102}=116$ ($(p,q)=(10,4)$).
By dividing into four sub-squares, we come back to the analysis of $(5,2)$. Observing that $5+2=7$ is odd
and that $2$ and $5$ are mutually prime, the general theory of Section~\ref{s7} applies
(see also Proposition~\ref{cor:antisymmetric}).
\end{enumerate}
\subsection{Remaining cases}
 In Section \ref{s5}, there were four cases for which we could not determine whether the Courant-sharp property holds.
This is due to the fact  their multiplicity is larger than 2 for $h=0$.
Three of these cases can be treated by the general theory of Section~\ref{s7} as we see below.
\begin{enumerate}
\item [11.] The case $\lambda_{78,h}=\lambda_{79,h}$ ($h>0$) corresponding to $(p,q) = (6,7)$.
Because $p$ and $q$ are mutually prime and $p+q$ is odd, the general theory of Section~\ref{s7} applies.
 \item[12.]  The case $\lambda_{126,h} =\lambda_{127,h} $($h>0$) corresponding to $(8,9)$. Hence $p$ and $q$ are mutually prime, $p+q$ is odd, and the general theory of Section~\ref{s7} applies.
 \item[13.]  The case $\lambda_{89,h}=\lambda_{90,h}$ ($h>0$) corresponding to $(6,8)$
  By dividing into four sub-squares, we come back to the analysis of
 $(3,4)$ which can be treated observing that $3$ and $4$ are mutually prime and that $3+4$ is odd
  (see also Proposition~\ref{cor:antisymmetric}).
 \end{enumerate}
 In all these cases, the multiplicity becomes $2$ for $h>0\,$.\\
 \subsection{Discussion}
We can treat the cases $(0,p)$ for $p>2$ by invoking Lemma~\ref{lem8.2} and the results of \cite{HPS1}
that the Neumann eigenvalues corresponding to these pairs are not Courant-sharp. The case $(0,1)$ was already discussed in Subsection~\ref{ss:2.3}. \\

 We also recall that we have completely analysed the cases $(p,p)$ for $p\geq 1$ in Section \ref{spp}.\\

 In conclusion, in order to achieve the proof of Theorem \ref{thm:hsmall}, it remains to analyse the following two cases:
  \begin{itemize}
 \item The case $(0,2)$ where our previous analysis does not permit us
 to exclude a possible Courant-sharp situation for $h>0$.
 \item The case $(p,q)=(7,9)$ where $p$ and $q$ are mutually prime but $p+q$ is even.
 \end{itemize}
 For the latter, some special analysis has to be done for $ \tan \theta = \frac79$.

 \subsection{Numerical illustration of the case $p=3$}\label{ss82}
We now illustrate and give a more detailed analysis of the typical case $(0,p)$ in the odd situation: $p=3$.
We recall that it was shown in \cite{HPS1} that the Neumann eigenvalue $\lambda_{10,0}$ corresponding to the pair $(0,3)$
is not Courant-sharp.\\
We set
\begin{equation*}
\Phi_{0,3}(x,y,\theta,h):= \Phi_{h,\theta,0,3} (x,y),
\end{equation*}
for $(x,y) \in (0,\frac{\pi}2)^2$.
Following the steps of the $(0,p)$ analysis, we see that the critical points only occur when $\theta = \frac \pi 4, \frac{3\pi}{4}$.
Below we plot $\Phi_{h,\theta,0,3} (x,y)$ for various values of $\theta$.

\begin{figure}[htp]
\centering
\includegraphics[width=\textwidth]{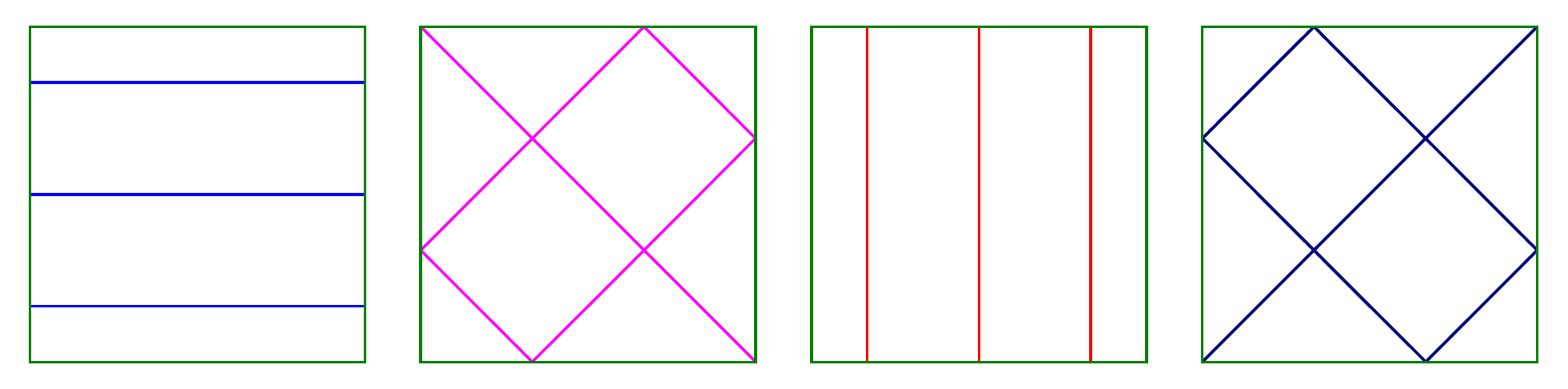}
\caption{The eigenfunction $\Phi_{h,\theta,0,3} (x,y)$ when $h=0$ for $\theta=0$ (blue), $\theta=\frac{\pi}{4}$ (magenta), $\theta =\frac{\pi}2$ (red) and $\theta=\frac{3\pi}4$ (navy).}
\label{fig:0-3N}
\end{figure}

From Figure~\ref{fig:0-3N}, we see that $\Phi_{h,\theta,0,3} (x,y)$ has 4 nodal domains except for the cases
$\theta = \frac \pi 4, \frac{3\pi}{4}$ when it has 8 nodal domains.

For $h$ sufficiently small, we have to concentrate on values of $\theta$ close to $\frac \pi 4$ and to look
at what happens in the neighbourhood of the boundary critical points.

By Remark~\ref{remsym}, since $p=3$ is odd, it is sufficient to consider $\theta =\frac \pi 4\,$.
For the Neumann case with $\theta = \frac \pi 4$, the critical points on the boundary are
$(\frac{\pi}2,-\frac{\pi}2)$, $(-\frac{\pi}2,\frac{\pi}2)$, $(\frac{\pi}6,\frac{\pi}2)$,
    $(\frac{\pi}2,\frac{\pi}6)$, $(-\frac{\pi}2,-\frac{\pi}6)$, $(-\frac{\pi}6,-\frac{\pi}2)$.\\
On the side $x=\frac{\pi}2$, we have that $\Phi_{h,\theta,0,3} (\frac{\pi}2,y)=0$ implies that
\begin{equation}\label{eq:bcpion21}
\tan \theta = -\frac{\cos (\alpha_0 /2)\sin (\alpha_3 y/\pi)}{\sin(\alpha_3 /2)\cos(\alpha_0 y /\pi)}\,.
\end{equation}
Similarly, we have that $\frac{\partial \Phi_{h,\theta,0,3}}{\partial y} (\frac{\pi}2,y)=0$ implies that
\begin{equation}\label{eq:bcpion22}
\tan \theta = \frac{\alpha_3 \cos (\alpha_0 /2)\cos (\alpha_3 y/\pi)}{\alpha_0 \sin(\alpha_3 /2)\cos(\alpha_0 y /\pi)}\,.
\end{equation}
By equating \eqref{eq:bcpion21} and \eqref{eq:bcpion22}, we obtain
\begin{equation}\label{eq:bcpion23}
\alpha_0 \tan\left(\frac{\alpha_0 y}{\pi}\right) = \alpha_3 \cot \left(\frac{\alpha_3 y}{\pi}\right).
\end{equation}
Let $y_c$ denote a solution of \eqref{eq:bcpion23} and set
\begin{equation}\label{eq:bcpion24}
\theta(h) = \arctan \left(-\frac{\cos (\alpha_0 /2)\sin (\alpha_3 y_c/\pi)}{\sin(\alpha_3 /2)\cos(\alpha_0 y_c /\pi)}\right).
\end{equation}
Then for $\theta=\theta(h)$, $(\frac{\pi}2, y_c)$ is a boundary critical point of $\Phi_{h,\theta,0,3} (x,y)$.

For h=0.01, we compute numerically that $y_c \approx 0.5236$.
Below we plot $ \Phi_{h,\theta,0,3}$ for $h=0.01$ and various values of $\theta$ in order to show the changes in the structure of the nodal domains.

 \begin{figure}
\centering
\subfloat[  $\theta=0$ (blue), $\theta=\theta(0.01)$ (orange). \label{10crit1}]{%
  \includegraphics[width=\textwidth]{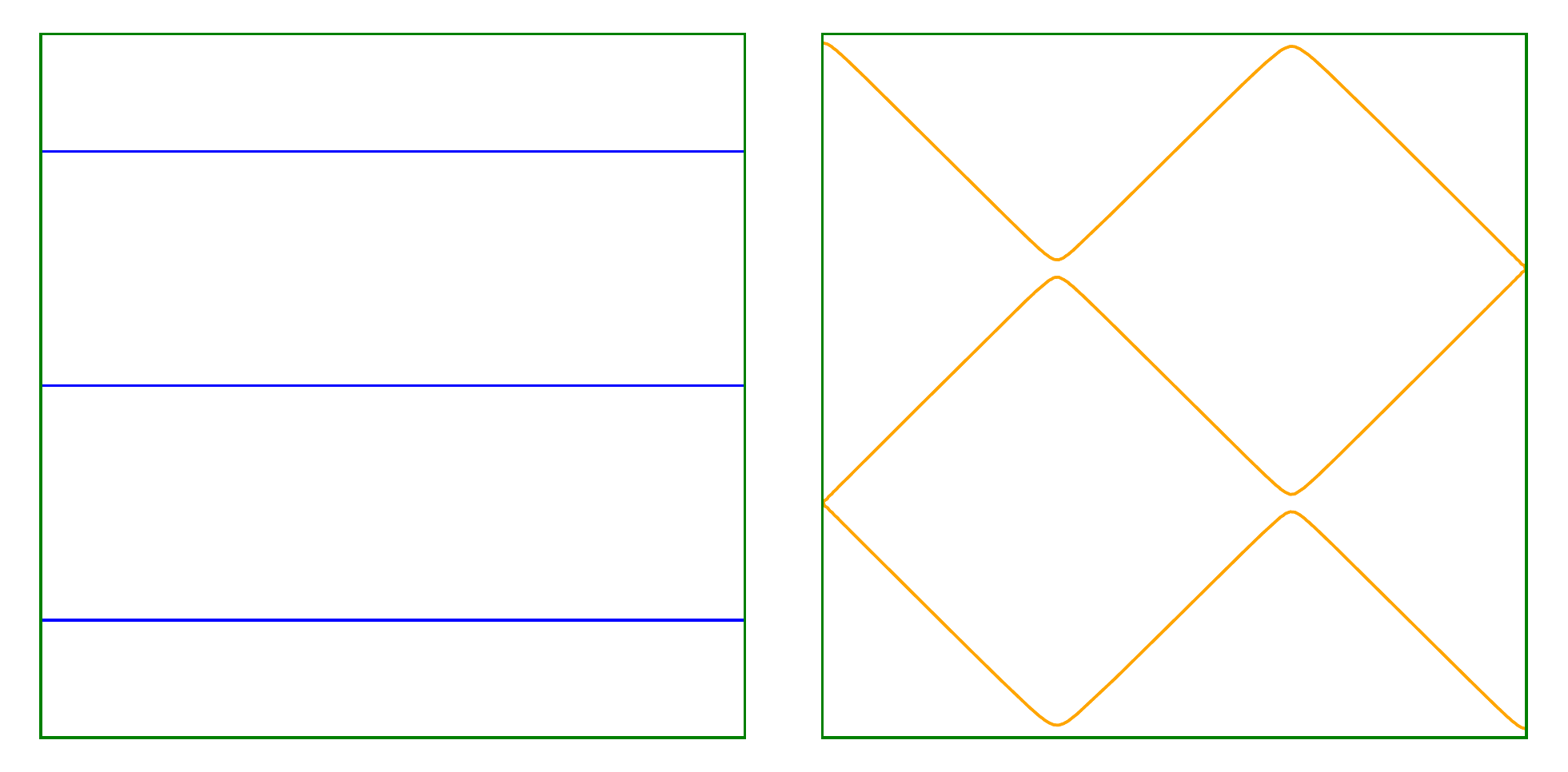}%
  }\par
\subfloat[  $\theta = \frac12 (\theta(0.01) + \frac{\pi}4)$ (purple), $\theta=\frac{\pi}{4}$ (magenta). \label{10crit2}]{%
  \includegraphics[width=\textwidth]{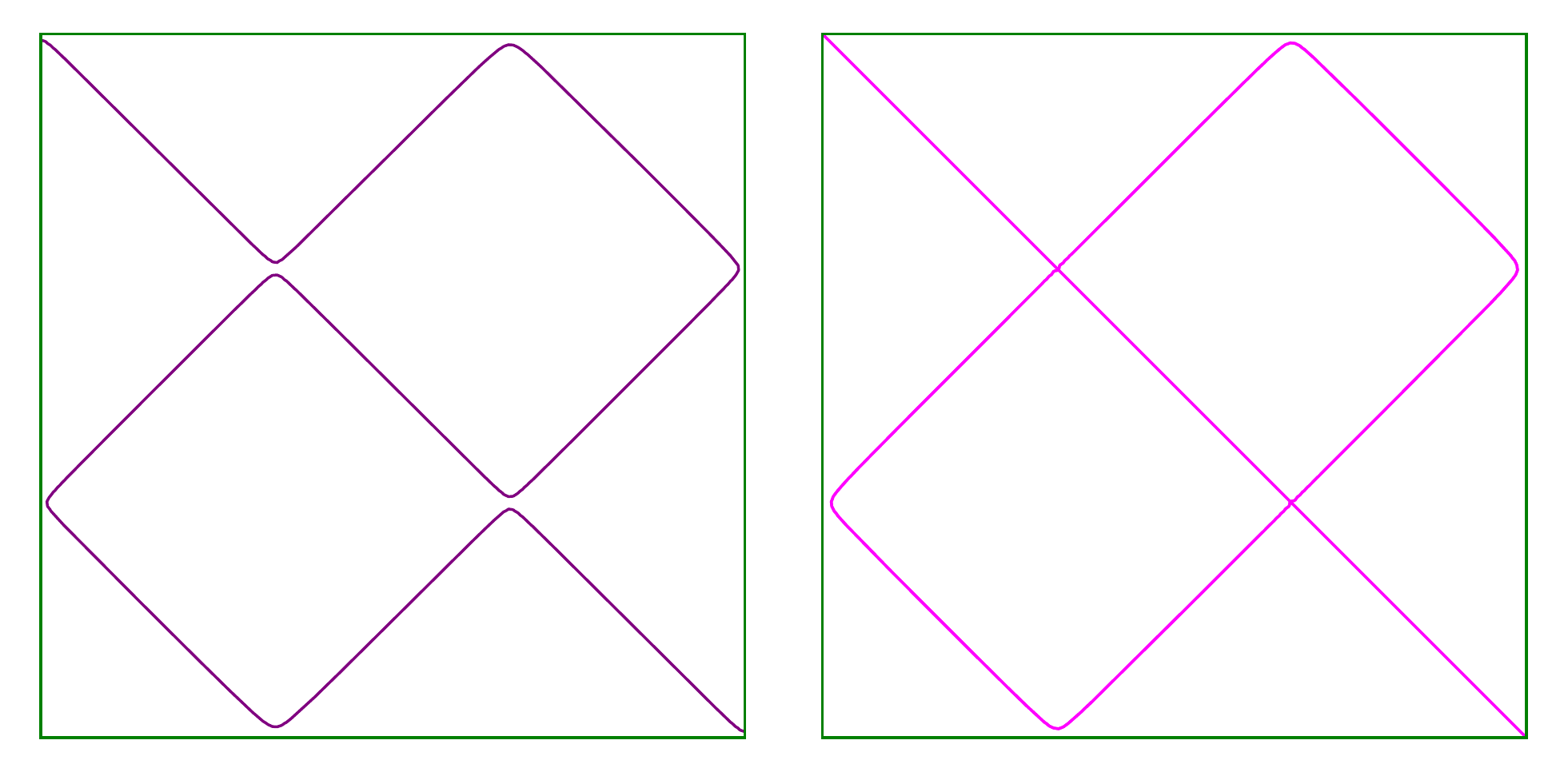}%
  }
\caption{The Robin eigenfunction $\Phi_{h,\theta,0,3}$ for $h=0.01$ and various values of $\theta$.
We note that $\theta(0.01)$ is given by Equation~\eqref{eq:bcpion24} with $h=0.01$.}
\label{fig:10h001}
\end{figure}

From Figure~\ref{fig:10h001} and the above analysis, we see that $ \Phi_{h,\theta,0,3}$ has
\begin{itemize}
\item 0 interior critical points, 6 boundary critical points and 4 nodal domains for $\theta \in [0,\theta(0.01))$,
\item 0 interior critical points, 4 boundary critical points for $\theta = \theta(0.01)$, and 4 nodal domains,
\item 0 interior critical points, 2 boundary critical points for $\theta \in (\theta(0.01),\frac{\pi}4)$, and 2 nodal domains,
\item 2 interior critical points, 2 boundary critical points for $\theta = \frac{\pi}4$, and 4 nodal domains.
\end{itemize}

\section{The case $(p,q)=(0,2)$}\label{s10}
The difficulty is that for $h=0$, we are in the Courant-sharp situation. We will show that the number of nodal domains
decreases  as $h>0$ (small) increases (the labelling  of the eigenvalue is constant in this case).

 We want to analyse the zero set of
 \begin{align*}
 \Phi_{h,\theta, 0,2} (x,y)&:=   \cos \theta \cos (\alpha_0(h) x/\pi)  \cos (\alpha_2(h) y /\pi) \\
 & + \sin \theta \cos (\alpha_2(h) x/\pi)  \cos (\alpha_0(h)  y/\pi)\,.
 \end{align*}
 We recall that the analysis for $h$ large was  obtained
 in \cite{GH}, but we are interested here with  the case where $h$  is small.\\
 We note that for $h=0$ this reads
  $$
 \Phi_{0,2} (x,y,\theta,0):=   \cos \theta   \cos (2 y)+ \sin \theta \cos (2x) \,.
 $$
 For this case, the only values of $\theta$ for which there are critical points are
$\theta = \frac \pi 4$ and $\frac{3\pi}{4}$. For $\theta =\frac \pi 4$, the critical points are only on the boundary
 at the middle of each side, that is $(\pm \frac{\pi}2,0), (0,\pm \frac{\pi}2)$. For $\theta=\frac{3 \pi}{4}$, there  is
 one interior critical point at $(0,0)$ and four boundary critical points at the corners $ (\pm \frac{\pi}2, \frac{\pi}2), (\pm \frac{\pi}2, -\frac{\pi}2)$. In all these cases, we are in a Morse situation.

To analyse the situation for $h$ small  on the boundary $x=-\frac{\pi}2$, we observe that we are still dealing with a Morse function $ \Phi_{h,\theta, 0,2}$ for $(\theta,h)$ close to $(\frac \pi 4,0)$, which admits a critical point $z_{h,\theta}$ close to $z_0 = (-\frac \pi 2,0)$. Hence we still have to consider the Morse picture and the corresponding zero-set.

 \paragraph{Analysis at the boundary for $\theta$ close to $\frac \pi 4$ and $h\geq 0$ close to $0$}~\\
    Here we can analyse the situation at the boundary $x=-\frac \pi 2$ and compare it with the situation on $y=-\frac \pi 2$. We can also deduce from
    \eqref{eq:2.8} that
    $$
    \Phi_{0,2} \left(-\frac \pi 2,y,\theta,h\right)= \Phi_{0,2} \left(y,-\frac \pi 2,\frac \pi 2- \theta,h\right)\,.
    $$
   Note that the understanding of what is going on for $\theta < \frac \pi 4$ on the first side gives the information of what is going on for $\theta >\frac \pi 4$ on the next side.
   In addition $ \Phi_{0,2} $ is symmetric with respect to the two axes. Hence, we can then use these symmetries  to understand the complete picture near the sides.

    For $h\geq 0$, we analyse the zeros of the restriction to $x=-\frac \pi 2$ of the eigenfunction which is given  by
   \begin{equation*}
 y\mapsto \psi(y,\theta,h):= \Phi_{h,\theta,0,2}(-\pi /2,y) \,.
  \end{equation*}
 \paragraph{The case $h=0\,$}~\\
 For $h=0$, this gives
   \begin{equation*}
  \psi (y,\theta,0)= \cos \theta   \cos (2 y) - \sin \theta =0  \,,
  \end{equation*}
  which takes the form
  \begin{equation}\label{eq:nodalset5h=0a}
  \cos (2 y )= \tan \theta\,.
  \end{equation}
  Note for later that
  \begin{equation}\label{crit}
    \partial_\theta \psi (0,\pi /4, 0) = -\sqrt{2} \,.
  \end{equation}
  When $\theta$ increases from $0$ to $\frac \pi 4$, we get immediately from \eqref{eq:nodalset5h=0a}  that there is a unique  $y(\theta)$ decreasing  from $y=\frac \pi 4$ to $0$ such that the zero set consists of $\pm y(\theta)$.
  For $\theta=\frac \pi 4$, $y(\frac \pi 4)$ is a non-degenerate critical point of $y\mapsto  \psi(y,\frac \pi 4,0)$. For $\theta >\frac \pi 4$, the eigenfunction has no zero on $x=-\frac \pi 2$. \\

 \paragraph{Application to the global nodal structure ($h=0$)} ~\\
  Using the above symmetries $(x,y,\theta) \mapsto (y,x,\frac \pi 2 -\theta)$, $ (x,y)\mapsto (x,-y)$ and  $ (x,y)\mapsto (-x,y)$ and the fact that there are no critical points inside the square  for $\theta \leq \frac{\pi}2$, we obtain that we start from $3$ nodal domains for $\theta < \frac \pi 4$, then get $5$ nodal domains at $\theta=\frac \pi 4$ before  obtaining $3$ nodal domains for $\theta > \frac \pi 4$. See Figure~\ref{fig:transitionh=0}. Hence we recall that we are in the Courant-sharp situation for $\theta=\frac \pi 4$.

\begin{figure}[h]
  \begin{center}
\includegraphics[width=12cm]{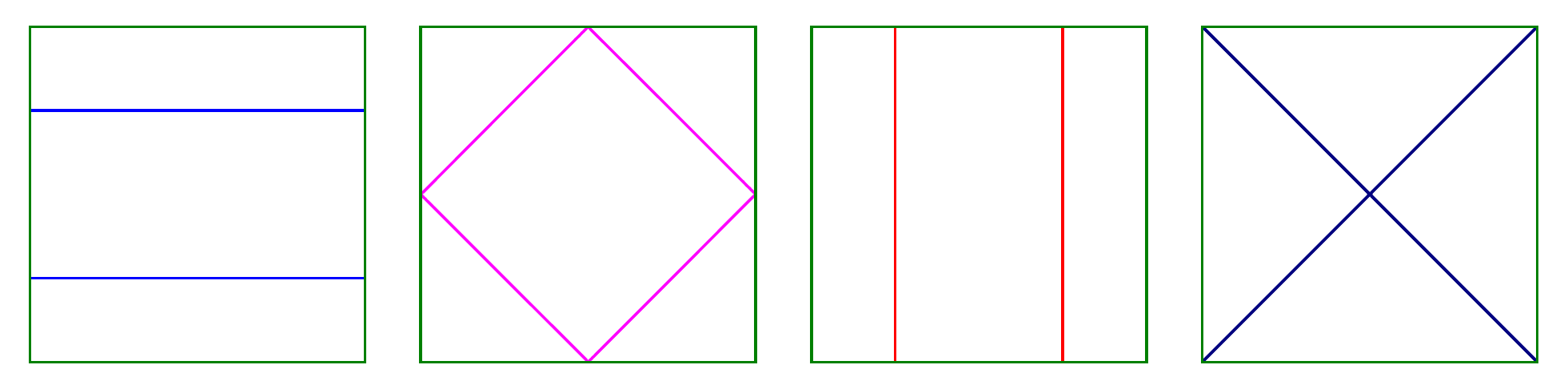}
 \caption{The fifth eigenfunction when $h=0$ for $\theta=0$ (blue), $\theta=\frac{\pi}{4}$ (magenta), $\theta =\frac{\pi}2$ (red) and $\theta=\frac{3\pi}4$ (navy).}
 \label{fig:transitionh=0}
  \end{center}
 \end{figure}

\paragraph{The case $h\geq 0\,$}~\\
For $\theta=\frac \pi 4$ and $h=0$, $y=0$ is a non-degenerate critical point of the Morse function $y \mapsto \psi(y,\frac \pi 4,0)$. On the other hand $y=0$ is always a critical point of $y\mapsto \psi (y,\theta,h)$. By the general properties of Morse functions depending on parameters, $y=0$ is locally the unique critical point
for $(\theta,h)$ close to $(\frac \pi 4,0)$ and is  also
non-degenerate.\\
The determination of the zeros is related to the sign of $\psi (0,\theta,h)$:  two solutions $\pm y(\theta,h)$ if $\psi (0,\theta,h) >0$, of course one double solution $y=0$ if $\psi (0,\theta,h) =0$ and no solution if $\psi (0,\theta,h) <0\,$.\\
For $y=0$ and $x=-\frac \pi 2$, we get
   \begin{equation*}
  \psi (0,\theta,h):= \cos \theta \cos (\alpha_0(h) /2)  + \sin \theta \cos (\alpha_2(h)/2)    \,.
  \end{equation*}
  Hence (in the neighbourhood of $(\frac \pi 4,0)$), due to the negativity of $\partial_\theta \psi$ (see \eqref{crit}), there is a unique $\theta(h)$ such that $ \psi (0,\theta(h),h)  =0$ and corresponding to a change of sign of $\theta \mapsto \psi (0,\theta,h)$.
   It remains to compute $\theta (h)$ for $h\geq 0$ small.
  $$
  \psi (0,\theta,h) = 0 \mbox{ if and only if }  \tan \theta = - \cos (\alpha_0(h)/2) / \cos (\alpha_2(h) /2) \,.
  $$
  We expand the right-hand side and, using the expansions (deduced from \eqref{eq:alphaneven})
  \begin{equation*}
  \alpha_0 (h) \sim \sqrt{2\pi h} \,,\, \alpha_2(h)-2\pi \sim h\,,
  \end{equation*}
  we obtain
  $$
  \tan \theta =  1 - \frac{\pi h}{4} + \mathcal O (h^2)\, ,
  $$
  so
  \begin{equation*}
  \theta(h) -\frac \pi 4 \sim -\frac{\pi h}{8}\,,
  \end{equation*}
  and in particular
  $\theta(h) < \frac \pi 4 $ for $h>0\,$.
  Hence, we get that for $\theta < \theta(h)$,  there are
  two zeros on the left side $ x=-\frac{\pi}2$ of the square (and consequently the same on the opposite side $ x=\frac{\pi}2$), one double zero for $\theta =\theta (h)$ and no zero for $\theta > \theta(h)$.\\
  Using the symmetry argument, we also get that on the two other sides  there are
  no zeros for $\theta < \frac \pi 2 - \theta(h) = \frac \pi 4  + \frac{\pi h}{8} + \mathcal O (h^2)$, a double zero for $\theta = \frac \pi 2 - \theta(h)$ and two zeros for $\theta > \frac \pi 2 - \theta(h)$.\\
  As soon as $h>0$ is small enough, we have  that
  $$
  \theta (h) < \frac \pi 2 -\theta (h)\,.
  $$
  Assuming that there are  no critical points inside the square, we obtain by topological considerations that the number of nodal domains is $3$ for $ 0 <\theta  \leq \theta(h)$, $2$ for $\theta(h) <\theta < \frac \pi 2 -\theta(h)$ and again $3$ for $\frac \pi 2 -\theta(h) \leq \theta < \frac \pi 2$. In particular, observing that the labelling of the eigenvalue is five, we cannot be  in the Courant-sharp situation for $h>0$ ($h$ small enough).

\paragraph{Analysis at the corner for $\theta$ close to $\frac{3\pi}{4}$}
  It remains to analyse the situation for $\theta$ close to $\frac{3\pi}{4} $ and $h\geq 0$.  We note that
   the critical point $(0,0)$ is stable and independent of $h$ and $\theta$ and the corresponding critical value is $\cos \theta +\sin \theta$. Hence it only appears for $\theta=\frac {3\pi} 4$.\\

  Considering the corners, say for example the corner $(-\frac \pi 2,\frac \pi 2)$, we again observe that, for any $h\geq 0$,  this is a critical point for $\theta =\frac{3\pi}{4}$.
  As we have seen, there are no other values of $\theta$ close to $\frac{3\pi}{4}$ such that we have a zero critical  point.
  It is not difficult to show that $ \Phi_{h,\theta,0,2}(x,y)$ satisfies the conditions \eqref{6.1c} for $z_0=(-\frac{\pi}2, \frac{\pi}2)$ and $\theta=\frac{3\pi}{4}$. Hence Proposition~\ref{Proposition6.2} applies
  and in a neighbourhood of the corner, there are exactly 2 nodal domains.

So the number of nodal domains  for $\frac{\pi}2 < \theta < \pi$ is $3$ or $4$.
In particular we have proved:
\begin{prop}
There exists $h_0 >0$ such that for $h\in (0,h_0)$ the eigenvalue $\lambda_{5,h}$ is not Courant-sharp.
In particular, there are three critical values $\theta_j^*(h) \in [0,\pi)$ $(j=1,2,3)$ such that
$$
 \theta_1^*(h) =  \arctan \left(-  \frac{\cos (\alpha_0 /2)}{\cos(\alpha_2 /2)}\right)\,, \,\theta_2^*(h) = \frac \pi 2 - \theta_1^*(h)\,, \,\theta_3^*= \frac {3\pi}{4}\,,
 $$
 and such that $ \Phi_{h,\theta,0,2}$ has:
 \begin{itemize}
\item  $3$ nodal domains for $\theta \in [0,\theta^*_1(h)]$;
\item $2$ nodal domains for $\theta \in (\theta_1^*(h),\theta^*_2(h))$;
\item $3$ nodal domains for  $\theta \in [\theta_2^*(h),\theta_3^*)$;
\item $4$ nodal domains for $\theta=\theta_3^*$;
\item $3$ nodal domains for $\theta\in (\theta_3^*,\pi)$.
\end{itemize}
   \end{prop}

 We remark the preceding result is the analogue of Proposition 6.1 of \cite{GH} for the case where
$h>0$ is large.

 In Figure~\ref{fig:5h001}, for $h=0.01$, we depict the transitions between the nodal partitions of $ \Phi_{h,\theta,0,2}(x,y)$ for $(x,y) \in (-\frac{\pi}2,\frac{\pi}2)^2$, as $\theta$ varies.

 \begin{figure}
\centering
\subfloat[  $\theta=0$ (blue), $\theta=\theta_1^*(0.01)$ (magenta), $\theta=\frac{\pi}4$ (red).\label{5crit1}]{%
  \includegraphics[width=\textwidth]{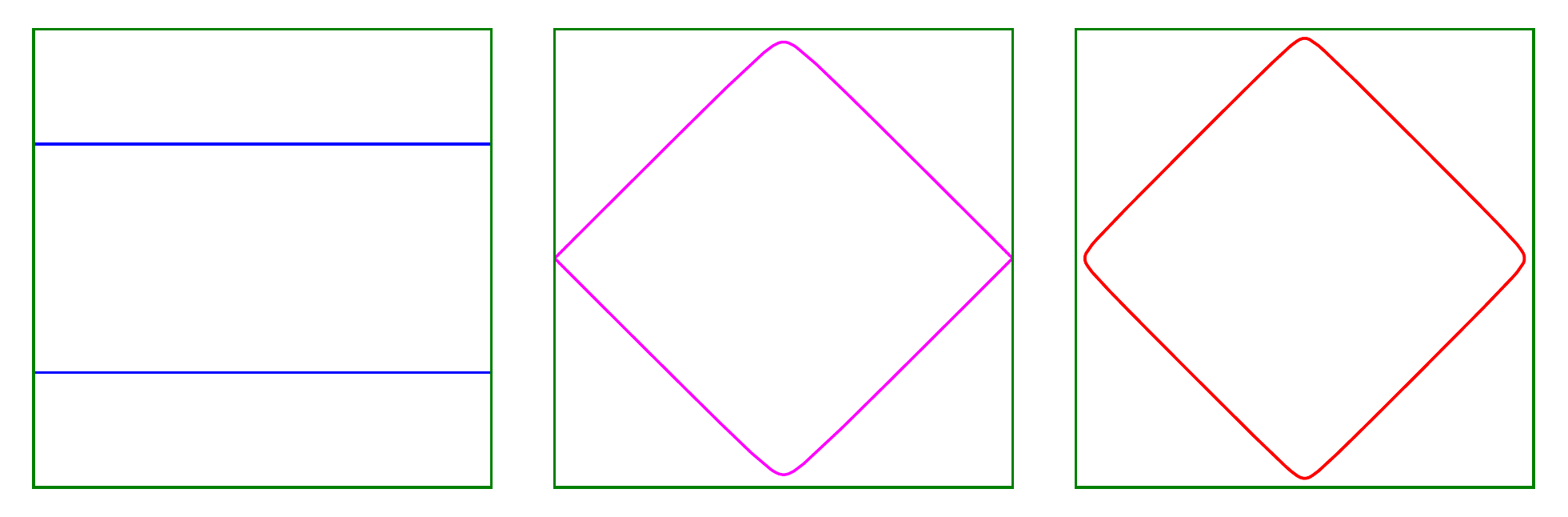}%
  }\par
\subfloat[  $\theta =\theta_2^*(0.01)$ (orange), $\theta=\frac{\pi}{2}$ (lime), $\theta=\frac{3\pi}{4}$ (navy). \label{5crit2}]{%
  \includegraphics[width=\textwidth]{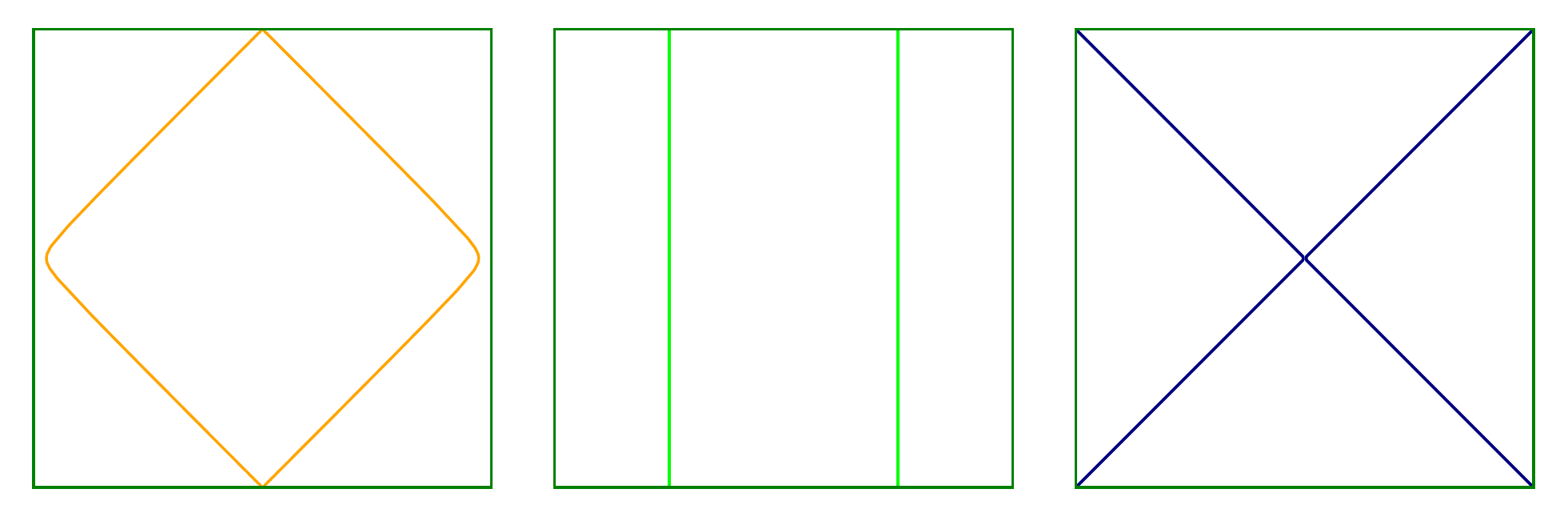}%
  }
\caption{The Robin eigenfunction $ \Phi_{h,\theta,0,2}$ for $h=0.01$ and various values of $\theta$.}
\label{fig:5h001}
\end{figure}
We observe that, by the preceding analysis, $ \Phi_{h,\theta,0,2}$ has
\begin{itemize}
\item 0 interior critical points and 4 boundary critical points for $\theta \in [0,\theta_1^*(h))$,
\item 0 interior critical points and 2 boundary critical points for $\theta = \theta_1^*(h)$,
\item 0 interior critical points and 0 boundary critical points for $\theta \in (\theta_1^*(h),\theta_2^*(h))$,
\item 0 interior critical points and 2 boundary critical points for $\theta = \theta_2^*(h)$,
\item 0 interior critical points and 4 boundary critical points for $\theta \in (\theta_2^*(h),\theta_3^*)$,
\item 1 interior critical point and 4 boundary critical points for $\theta = \theta_3^*$,
\item 0 interior critical points and 4 boundary critical points for $\theta \in (\theta_3^*,\pi)$.
\end{itemize}

\section{The case $(p,q)=(7,9)$}\label{s11}

On $\hat{S} = (0,\pi)^2$, the Robin eigenfunction for $h \geq 0$ corresponding to $(p,q)=(7,9)$ is
\begin{align*}
\Phi_{7,9,\theta,h}(x,y) &= \cos\theta \sin \left(\frac{\alpha_7(h) x}{\pi} - \frac{\alpha_7(h)}2\right)
\sin\left(\frac{\alpha_9(h) y}{\pi} - \frac{\alpha_9(h)}2\right) \\
& \ \ \ \ \ \ +
\sin\theta \sin\left(\frac{\alpha_9 (h)x}{\pi} - \frac{\alpha_9(h)}2\right)
\sin \left(\frac{\alpha_7(h) y}{\pi} - \frac{\alpha_7(h)}2\right).
\end{align*}

We focus on the value of $\theta$ for which $\tan\theta = \frac79$ as our preceding analysis does
not apply in this case (see Remark~\ref{rem:8.4}).

By Sturm's theorem, on any boundary edge of $\hat{S}$, $\Phi_{7,9,\theta,h}$ has at least 7 zeros and at most 9 zeros.
In particular, for $h=0$, we see that on a given side, say $y=\pi$, there are 6 zeros where the nodal set meets
the boundary $\partial \hat{S}$ transversally (see  the central figure in Figure~\ref{fig:79}). Standard Morse theory applies in a neighbourhood of each of these zeros so each such critical point is isolated under a small perturbation of $h$.

\begin{figure}[htp]
\centering
\includegraphics[width=\textwidth]{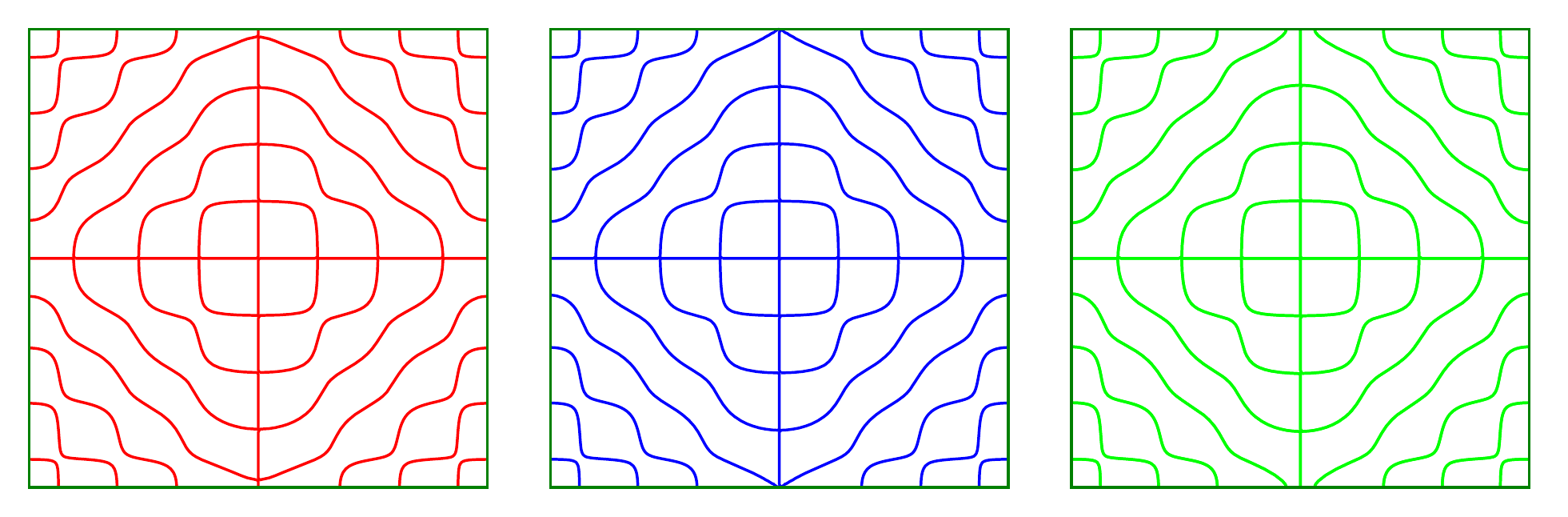}
\caption{The nodal set of the Neumann eigenfunction corresponding to $(p,q)=(7,9)$
 for $\theta = \frac{13\pi}{64}$ (red curve), $\theta = \arctan(\frac79)$ (blue curve) and $\theta = \frac{14\pi}{64}$
(green curve).}
\label{fig:79}
\end{figure}

For the point $ z_0 = (\frac{\pi}2, 0)$, we must analyse how the nodal set in the neighbourhood $B(z_0,\varepsilon_0)$ changes under a small perturbation of $h$.
As we increase $h>0$ (small), by Sturm's theorem, we either obtain:
\begin{enumerate}
\item[(i)] 2 additional zero critical points on $y=0\,$,
\item[(ii)] 1 additional zero critical point on $ y=0\,$,
\item[(iii)] no additional zero critical points on $ y=0\,$.
\end{enumerate}

We observe that $x=\frac{\pi}2$ belongs to the nodal set for any $\theta$ and $y \in (0,\pi)$.
We also see that $\Phi_{7,9,\theta,h}(\pi-x,y) = -\Phi_{7,9,\theta,h}(x,y)$. So if $(x,y)$
belongs to the nodal set of $\Phi_{7,9,\theta,h}(x,y)$ then so does $(\pi-x,y)$.
This eliminates possibility (ii).

We introduce the following notation to facilitate our discussion.

Let $\hat{B}(z_0,\epsilon_0)=B(z_0,\epsilon_0) \cap \hat{S}$.
Define $\partial_{B} \hat{B}:= \partial \hat{B}(z_0,\epsilon_0) \cap \partial \hat{S}$, which is an interval, and $\partial_{I} \hat{B}:= \partial \hat{B}(z_0,\epsilon_0) \setminus \partial_{B} \hat{B}(z_0,\epsilon_0)$, which is a semi-circle.

By Lemma 5.3 of \cite{GH}, we can choose $\varepsilon_0$ small enough such that via a small perturbation of $h$,
it is not possible to obtain a nodal domain $\omega \subset \hat{B}(z_0,\epsilon_0)$ such that $\partial \omega \cap \partial_{B} \hat{B}$ consists of at most finitely many points.
We are not able to exclude the possibility that via a small perturbation of $h$, a nodal domain $\omega$ is obtained
such that $\partial \omega \cap \partial_{B} \hat{B}$ is a non-trivial interval. This is because we start from the
Neumann case $h=0$ for which Lemma 5.3 of \cite{GH} does not apply.

By the local structure of a Neumann eigenfunction, when $h=0$ there are 3 distinct nodal lines emanating
from $z_0$ and intersecting $\partial_{I} \hat{B}$ in 3 distinct points. This gives rise to 4 nodal domains.
After a small perturbation of $h$, there are still 3 points that belong to the intersection of the nodal set
with $\partial_{I} \hat{B}$.

Let $N$ denote the nodal set.
If $N \cap \partial_{I} \hat{B}$ contains 3 points and $N \cap \partial_{B} \hat{B}$ contains 3 points,
then $\hat{B}(z_0,\varepsilon_0)$ contains at most 6 nodal domains.
In this case, there are at most two nodal domains whose boundaries intersect $\partial_{B} \hat{B}$ on a
non-trivial interval and do not intersect $\partial_{I} \hat{B}$.

If $N \cap \partial_{I} \hat{B}$ contains 3 points and $N \cap \partial_{B} \hat{B}$ contains 1 point,
then $\hat{B}(z_0,\varepsilon_0)$ contains at most 4 nodal domains (there are no other possibilities in this case
since $\epsilon_0$ was chosen small enough above so that Lemma 5.3 of \cite{GH} applies).

Hence, after a small perturbation of $h$, we gain at most two  additional nodal domains in $\hat{B}(z_0,\varepsilon_0)$. Taking into account that we have two such boundary critical points, the maximum number
of  additional nodal domains that we gain under a small perturbation of $h$ is 4.

For $h=0$, $\Phi_{7,9,\theta,0}$ has 32 nodal domains. So for $h$ small, $\Phi_{7,9,\theta,h}$
has at most 36 nodal domains. But the pair $(7,9)$ corresponds to $\lambda_{116,h}$, so $\lambda_{116,h}$
is not Courant-sharp for $h$ sufficiently small.

 In Figure~\ref{fig:79v2} we plot the nodal set of $\Phi_{7,9,\theta,h}(x,y)$ for $(x,y)$ in a neighbourhood of $(\frac{\pi}2,0)$, $h=0.1$ and $\theta = \frac{13\pi}{64}, \arctan(\frac79), \frac{14\pi}{64}$ respectively.
 We see that for $h=0.1$, the number of nodal domains in $\hat{B}(z_0,\varepsilon_0)$ is at most 4 as $\theta$
close to $\arctan(\frac79)$ varies.

\begin{figure}[htp]
\centering
\includegraphics[width=\textwidth]{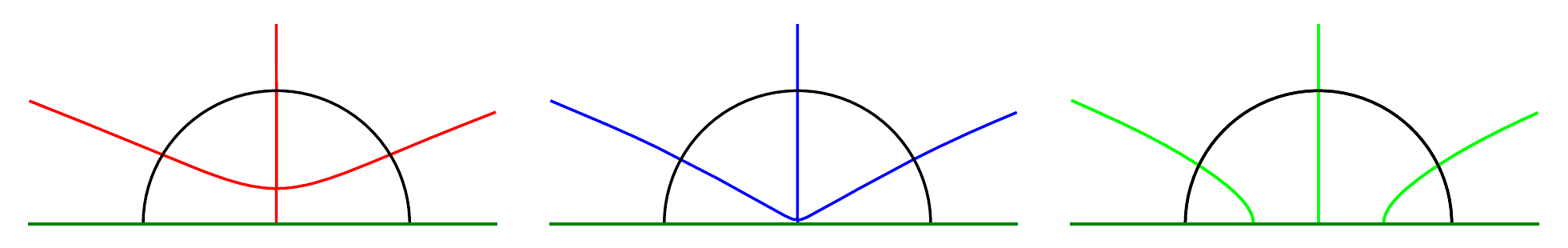}
\caption{The nodal set of the Robin eigenfunction corresponding to $(p,q)=(7,9)$ with $h=0.1$
for $\theta = \frac{13\pi}{64}$ (red curve), $\theta = \arctan(\frac79)$ (blue curve) and $\theta = \frac{14\pi}{64}$
(green curve). We are interested in the behaviour near the zero critical point $(\frac \pi 2,0)$. }
\label{fig:79v2}
\end{figure}

In Figure~\ref{fig:79szoom}, we focus on the case $\theta=\arctan(\frac79)$. We see that for $h=0$, $(\frac{\pi}2,0)$ is a triple point, whereas for $h=0.1$ we no longer have a triple point.

\begin{figure}[htp!]
\centering
\includegraphics[width=5.5cm]{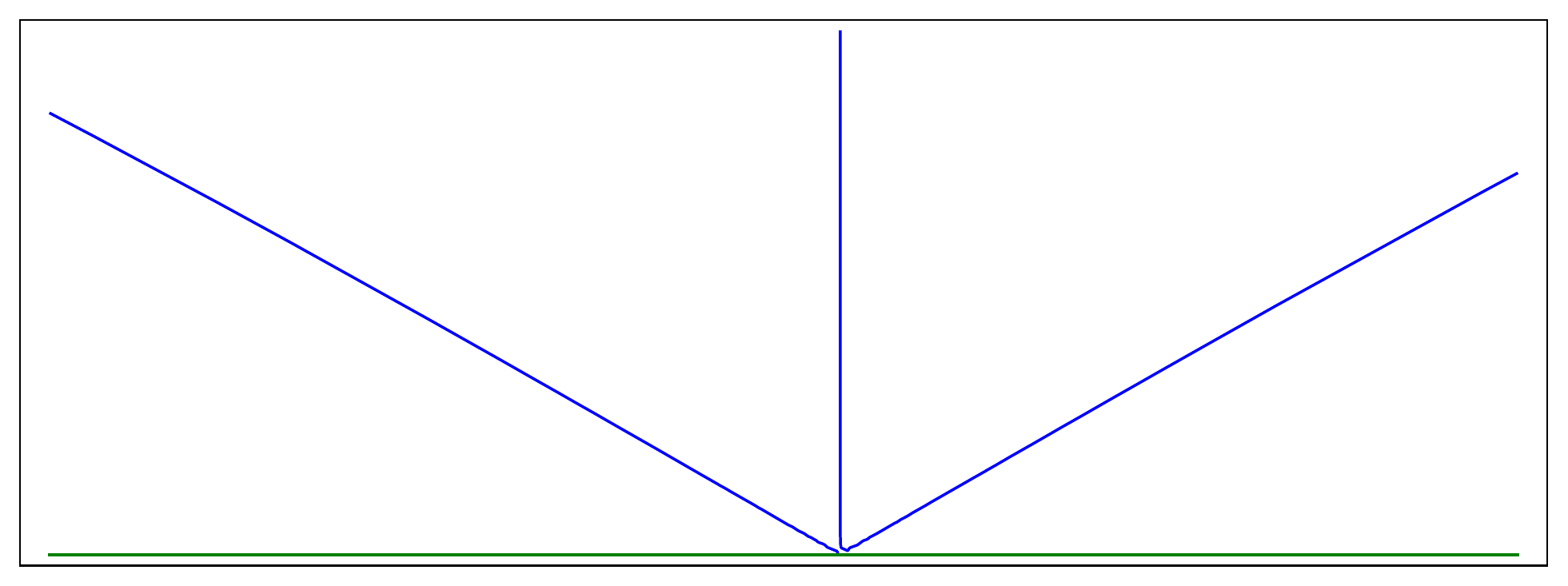}
\hskip 1cm
\includegraphics[width=5.5cm]{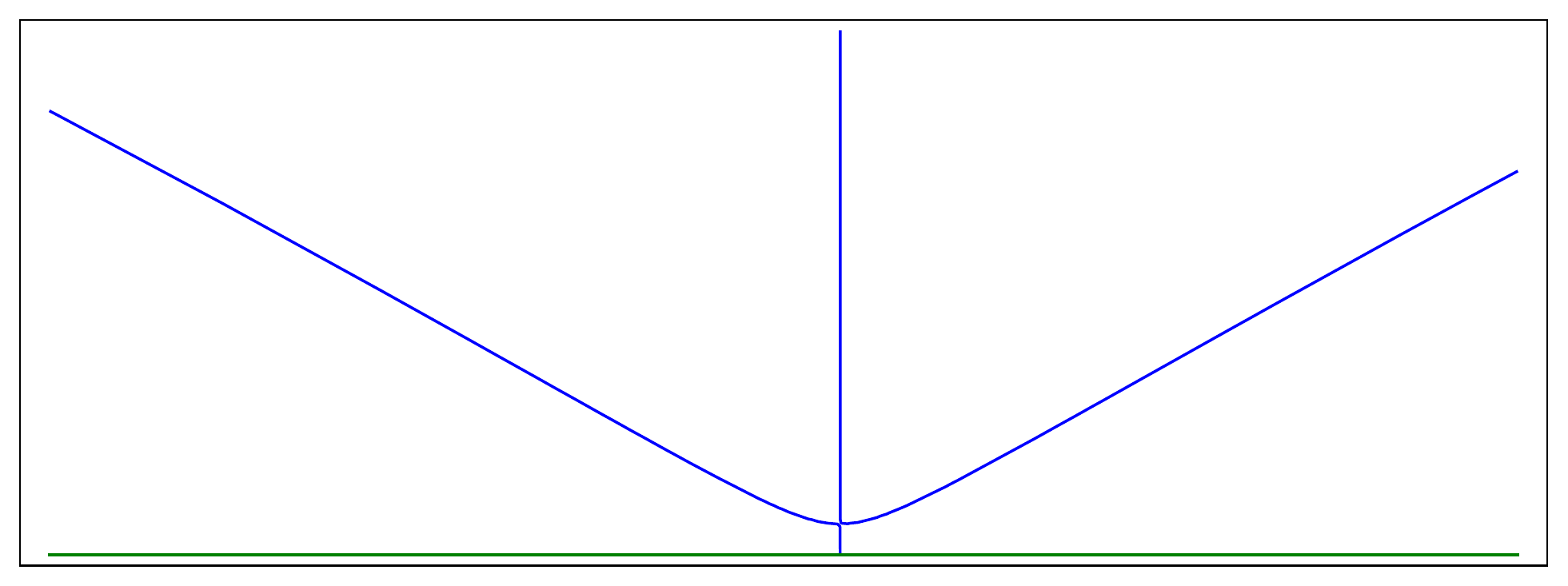}
\caption{The nodal set of the Neumann eigenfunction (left), and the Robin eigenfunction with $h=0.1$ (right) corresponding to $(p,q)=(7,9)$ and $\theta=\arctan(\frac79)$.}
\label{fig:79szoom}
\end{figure}

\newpage

\appendix

\section{First Neumann eigenvalues of a square}\label{sA}

\begin{tabular}{|c  c  c  c |}
\hline
Neumann & & & \\
\hline
$m$ & $n$ & $m^2+n^2$ & $k$\\
\hline
0 & 0 & 0 & 1\\
1 &	0 & 1 &	2,3\\
0 & 1 & 1 & 2,3\\
1 &	1 &	2 &	4\\
2 &	0 &	4 &	5,6\\
0 &	2	& 4 &	5,6\\
2	& 1 &	5 &	7,8\\
1 &	2 &	5 &	7,8\\
2 &	2 &	8 &	9\\
3 &	0 &	9 &	10,11\\
0 &	3 &	9 &	10,11\\
3 &	1 &	10 &	12,13\\
1 &	3 &	10 &	12,13\\
3 &	2 &	13 &	14,15\\
2 &	3 &	13 &	14,15\\
4 &	0 &	16 &	16,17\\
0 &	4 &	16 &	16,17\\
4 &	1 &	17 &	18,19\\
1 &	4 &	17 &	18,19\\
3 &	3 &	18 &	20\\
4 &	2 &	20 &	21,22\\
2 &	4 &	20 &	21,22\\
5 &	0 &	25 &	23,24,25,26\\
0 &	5 &	25 &	23,24,25,26\\
4 &	3 &	25 &	23,24,25,26\\
3 &	4 &	25 &	23,24,25,26\\
5 &	1 &	26 &	27,28\\
1 &	5 &	26 &	27,28\\
5 &	2 &	29 &	29,30\\
2 &	5 &	29 &	29,30\\
4 &	4 &	32 &	31\\
5 &	3 &	34 &	32,33\\
3 &	5 &	34 &	32,33\\
6 &	0 &	36 &	34,35\\
0 &	6 &	36 &	34,35\\
6 &	1 &	37 &	36,37\\
1 &	6 &	37 &	36,37\\
6 &	2 &	40 &	38,39\\
2 &	6 &	40 &	38,39\\
5 &	4 &	41 &	40,41\\
4 &	5 &	41 &	40,41\\
6 &	3 &	45 &	42,43\\
3 &	6 &	45 &	42,43\\
7 &	0 &	49 &	44,45\\
0 &	7 &	49 &	44,45\\
\hline
\end{tabular}
\quad
\begin{tabular}{|c  c  c  c |} 
\hline
Neumann & & & \\
\hline
$m$ & $n$ & $m^2+n^2$ & $k$\\
\hline
7 &	1 &	50 &	46,47,48\\
5 &	5 &	50 &	46,47,48\\
1 &	7 &	50 &	46,47,48\\
6 &	4 &	52 &	49,50\\
4 &	6 &	52 &	49,50\\
7 &	2 &	53 &	51,52\\
2 &	7 &	53 &	51,52\\
7 &	3 &	58 &	53,54\\
3 &	7 &	58 &	53,54\\
6 &	5 &	61 &	55,56\\
5 &	6 &	61 &	55,56\\
8 &	0 &	64 &	57,58\\
0 &	8 &	64 &	57,58\\
8 &	1 &	65 &	59,60,61,62\\
1 &	8 &	65 &	59,60,61,62\\
7 &	4 &	65 &	59,60,61,62\\
4 &	7 &	65 &	59,60,61,62\\
8 &	2 &	68 &	63,64\\
2 &	8 &	68 &	63,64\\
6 &	6 &	72 &	65\\
8 &	3 &	73 &	66,67\\
3 &	8 &	73 &	66,67\\
7 &	5 &	74 &	68,69\\
5 &	7 &	74 &	68,69\\
8 &	4 &	80 &	70,71\\
4 &	8 &	80 &	70,71\\
9 &	0 &	81 &	72,73\\
0 &	9 &	81 &	72,73\\
9 &	1 &	82 &	74,75\\
1 &	9 &	82 &	74,75\\
9 &	2 &	85 &	76,77,78,79\\
2 &	9 &	85 &	76,77,78,79\\
7 &	6 &	85 &	76,77,78,79\\
6 &	7 &	85 &	76,77,78,79\\
8 &	5 &	89 &	80,81\\
5 &	8 &	89 &	80,81\\
9 &	3 &	90 &	82,83\\
3 &	9 &	90 &	82,83\\
9 &	4 &	97 &	84,85\\
4 &	9 &	97 &	84,85\\
7 &	7 &	98 &	86\\
10 & 0 & 100 &	87,88,89,90\\
0 &	10 & 100 &	87,88,89,90\\
8 &	6 &	100	& 87,88,89,90\\
6 &	8 &	100	& 87,88,89,90\\
\hline
\end{tabular}

\begin{tabular}{|c  c  c  c |} 
\hline
Neumann & & & \\
\hline
$m$ & $n$ & $m^2+n^2$ & $k$\\
\hline
10 & 1 & 101 &	91,92\\
1 &	10 & 101 &	91,92\\
10 & 2 & 104 & 93,94\\
2 &	10 & 104 & 93,94\\
9 & 5 & 106 & 95,96\\
5 &	9 &	106 & 95,96\\
10 & 3 & 109 & 97,98\\
3 &	10 & 109 & 97,98\\	
8 &	7 &	113 & 99,100\\
7 &	8 &	113 & 99,100\\
10 & 4 & 116 & 101,102\\
4 &	10 & 116 & 101,102\\	
9 &	6 &	117 & 103,104\\
6 &	9 &	117 & 103,104\\	
11 & 0 & 121 & 105,106\\
0 &	11 & 121 & 105,106\\	
11 & 1 & 122 & 107,108\\
1 &	11 & 122 & 107,108\\	
11 & 2 & 125 & 109 - 112\\
2 &	11 & 125 & 109 - 112\\	
10 & 5 & 125 & 109 - 112\\	
5 &	10 & 125 & 109 - 112\\	
8 &	8 &	128 & 113\\
11 & 3 & 130 & 114 - 117\\
3 &	11 & 130 & 114 - 117\\	
9 &	7 &	130 & 114 - 117\\	
7 &	9 &	130 & 114 - 117\\	
10 & 6 & 136 & 118,119\\
6 &	10 & 136 & 118,119\\
11 & 4 & 137 & 120,121\\
4 &	11 & 137 & 120,121\\	
12 & 0 & 144 & 122,123\\
0 &	12 & 144 & 122,123\\
12 & 1 & 145 & 124 - 127\\
9 &	8 &	145 & 124 - 127\\	
8 &	9 &	145 & 124 - 127\\	
1 &	12 & 145 & 124 - 127\\	
11 & 5 & 146 & 128,129\\
5 &	11 & 146 & 128,129\\	
\hline
\end{tabular}

\end{document}